\documentclass[11pt]{amsart}
\usepackage{graphicx,pinlabel}
\usepackage[all]{xy}
\usepackage[mathscr]{eucal}
\usepackage{microtype}
\usepackage{amssymb}
\usepackage[parfill]{parskip}
\usepackage{color}
\usepackage[top=1in, bottom=1in, left=1.5in, right=1.5in]{geometry}
\usepackage{lpic}
\usepackage{caption,subcaption}
\usepackage{hyperref}
\usepackage{cancel, xcolor, paralist, soul, stmaryrd}

\definecolor{green3}{HTML}{00B000}

\def\cn{{\mathcal N}}
\def\bF{{\mathbb F}}
\def\sp{{\mathrm{Sp}}}

\newcommand{\br}{\mathbb{R}}

\newcommand{\bz}{\mathbb Z}
\newcommand{\bd}{\mathbb D}

\newcommand{\cs}{\mathcal S} 
\newcommand{\cc}{\mathcal C} 
\newcommand{\cf}{\mathcal F} 
\newcommand{\ci}{\mathcal I} 
\newcommand{\ck}{\mathcal K} 
\newcommand{\ca}{\mathcal A} 

\newcommand{\lk}{\mathrm{Lk}} 
\newcommand{\Mod}{\mathrm{Mod}} 
\newcommand{\PMod}{\mathrm{PMod}} 

\newcommand{\Hom}{\mathrm{Hom}} 
\newcommand{\Sym}{\mathrm{Sym}} 
\newcommand{\Aut}{\mathrm{Aut}} 

\newcommand{\into}{\hookrightarrow}
\newcommand{\del}{\partial}

\DeclareMathOperator{\cg}{CG}
\DeclareMathOperator{\kg}{KG}

\newtheorem{Thm}{Theorem}[section]
\newtheorem{Prob}[Thm]{Problem}
\newtheorem{Prop}[Thm]{Proposition}
\newtheorem{Lem}[Thm]{Lemma}
\newtheorem{Cor}[Thm]{Corollary}

\theoremstyle{definition}

\theoremstyle{remark}
\newtheorem{Rem}[Thm]{Remark}

\numberwithin{equation}{section}

\title{Coloring curves on surfaces}

\author{Jonah Gaster}
\author{Joshua Evan Greene}
\author{Nicholas G. Vlamis}

\address{Department of Mathematics, Boston College \\ Chestnut Hill, MA 02467}

\email{gaster@bc.edu,joshua.greene@bc.edu}

\address{Department of Mathematics, University of Michigan \\ Ann Arbor, MI 48109}

\email{vlamis@umich.edu}

\begin{document}

\maketitle

\begin{abstract}
We study the chromatic number of the curve graph of a surface.
We show that the chromatic number grows like $k \log k$ for the graph of separating curves on a surface of Euler characteristic $-k$.
We also show that the graph of curves that represent a fixed non-zero
homology class is uniquely $t$-colorable, where $t$ denotes its clique number.
Together, these results lead to the best known bounds on the chromatic number of the curve graph.
We also study variations for arc graphs and obtain exact results for surfaces of low complexity.
Our investigation leads to connections with Kneser graphs, the Johnson homomorphism, and hyperbolic geometry.
\end{abstract}


\section{Introduction}

Curve graphs play a central role in the study of mapping class groups, Teichm\"{u}ller spaces, and 3-manifolds.
In this setting, their large-scale geometry has grown into a subject of intensive study \cite{Masurgeometry1,Bowditchtight,Minskyclassification,Brockclassification,Bestvinaquasitree}.
Alongside it, interest has grown in their graph-theoretic properties \cite{Ivanovautomorphisms,Aramayonafinite,MalesteinDesigns,Birmandeadends,Kim-Koberda}.
Here we explore their {\em chromatic number}, one of the most natural and attractive invariants in graph theory.

We briefly fix some terminology.
Let $S$ denote a compact, connected, orientable surface.  
The \emph{curve graph} $\cc(S)$ is the graph whose vertices are isotopy classes of essential simple closed curves on $S$, where two isotopy classes are adjacent if they have disjoint representatives \cite{Harveyboundary}.
We refer to the vertices of $\cc(S)$ simply as $\emph{curves}$.
The {\em chromatic number} $\chi(G)$ of a graph $G$ is the fewest number of colors required to color the vertices of $G$ so that adjacent vertices get different colors.
Thus, our motivating problem is to estimate $\chi(\cc(S))$, the fewest number of parts required to partition the curves on $S$ so that any two curves in a given part intersect.

Bestvina, Bromberg, and Fujiwara were the first to study the quantity $\chi(\cc(S))$.
They showed that it is finite en route to proving that the mapping class group $\Mod(S)$ has finite asymptotic dimension \cite[Lemma 5.6]{Bestvinaquasitree}.
Curve graphs are locally infinite, so the finite colorability is not at all apparent \emph{a priori}.
Their bound on $\chi(\cc(S_g))$ for a closed surface of genus $g$ is doubly-exponential in $g$, which they did not attempt to optimize.
By contrast, a simple lower bound on $\chi(G)$ comes from its \emph{clique number} $\omega(G)$, the size of the largest complete subgraph.
Maximum cliques in $\cc(S)$ correspond to pants decompositions of $S$, so $\omega(\cc(S_g)) = 3g-3$ for $g \ge 2$.
Our work was motivated in part to close the gap between the bounds $\Omega(g) = \chi(\cc(S_g)) = O(\exp(\exp(g)))$. (See \S\ref{subsec: notation} for the notation $\Omega$, $O$, and $\Theta$.)

Another source of motivation comes from the study of {\em topological designs} \cite{Juvansystems}.
An attractive unsolved problem in this area is to determine the size of a largest \emph{1-system} on $S_g$, i.e.~a collection of simple closed curves that pairwise intersect at most once.
Denote this value $N(g)$.
Malestein, Rivin, and Theran proved that $\Omega(g^2) = N(g) = O(\exp(g))$ and that the size of a largest collection of curves that pairwise intersect {\em exactly} once is $2g+1$ \cite[Thms.~1.1\&1.4]{MalesteinDesigns}. 
P. Przytycki dramatically improved the upper bound to $N(g) = O(g^3)$ \cite[Thms.~1.2\&1.4]{PrzytyckiArcs}.
A color class in a proper coloring of the subgraph of $\cc(S_g)$ induced on a $1$-system is precisely a collection of curves that pairwise intersect exactly once.
It follows that $N(g) \le (2g+1) \chi(\cc(S_g))$.
Therefore, a sub-quadratic upper bound on $\chi(\cc(S_g))$ would improve on the best current upper bound on $N(g)$;
conversely, large constructions of 1-systems would lead to improved lower bounds on $\chi(\cc(S_g))$.


\subsection{Prelude.}
\label{subsec: prelude}
We set the stage for our results with a brief, informal account of the case of the $n$-holed sphere $\Sigma_n$.
A curve on $\Sigma_n$ partitions its holes into two non-empty parts.
Disjoint curves give rise to \emph{nested} partitions: one part for one curve is properly contained in one part for the other curve.
We thereby obtain a \emph{homomorphism} from $\cc(\Sigma_n)$ to a finite graph $\kg(n)$ that records partitions of the holes and the nesting relation.
On the other hand, arranging the holes of $\Sigma_n$ around a circle, we obtain a finite subgraph $\cg(n)$\footnote{For convenience, in \S \ref{subsec: prelude}, we let $\cg(n)$ denote the graph $\cg(n) \smallsetminus \cg(n,1)$ defined in \S \ref{sec: total kneser}.} 
of those curves that surround a cyclic interval of holes. 
We thus obtain a sequence of homomorphisms $\cg(n) \to \cc(\Sigma_n) \to \kg(n)$, the first of which is an embedding.
Homomorphisms compose, and a $t$-coloring is a homomorphism with target a complete graph $K_t$.
It follows that $\chi(\cg(n)) \le \chi(\cc(\Sigma_n)) \le \chi(\kg(n))$.
The graphs $\kg(n)$ and $\cg(n)$ are natural unions of the well-studied two-parameter families of Kneser graphs $\kg(n,k)$ and cyclic interval graphs $\cg(n,k)$. 
As we show, their chromatic numbers both grow like $n \log n$, and in fact are within a factor of $\ln(2) \approx 0.69$ apart.
In this way, we obtain a very precise estimate on $\chi(\cc(\Sigma_n))$.

The graph $\cg(n)$ repeatedly arises as a subgraph of arc and curve graphs, as does $\kg(n)$ as a target for homomorphisms from these graphs.
The tight control we gain over the (fractional) chromatic numbers of $\cg(n)$ and $\kg(n)$ underpins many of our results for curve graphs.
Other researchers have used the graph $\cg(n)$ as a probe of the large-scale geometry of the curve complex $\cc(\Sigma_n)$.
Embeddings of $\cg(n)$ into $\cc(\Sigma_n)$ are \emph{rigid} in the sense that any two are related by an automorphism of $\cc(\Sigma_n)$ \cite{Aramayonafinite}.
The induced subcomplex on $\cg(n)$ is homeomorphic to an $(n-4)$-dimensional sphere  \cite{Leeassociahedron}, and it represents a generator of the homology of $\cc(\Sigma_n)$ as a $\Mod^{\pm}(\Sigma_n)$-module \cite{Birmanfinite}.
Thus, it is remarkable -- or not? -- that it also accounts for the order of growth of $\chi(\cc(\Sigma_n))$.
We note that the {\em maximal clique graph} $\ck(\cg(n))$ (see \S\ref{subsec: notation}) is isomorphic to the \emph{associahedron} or the \emph{flip graph on the triangulations of an $n$-gon}.
The determination of the chromatic number of $\chi(\ck(\cg(n)))$ is a fascinating open problem: it is not even known whether it grows unbounded in $n$ \cite[\S\S4\&6]{FabilaChromatic}, \cite{Sleatorrotation}.


\subsection{Statement of results.}
\label{subsec: results}

In \S \ref{sec: total kneser}, we introduce the total Kneser graph $\kg(n)$ and the total cyclic interval graph $\cg(n)$.
In spite of how naturally they arise in the present setting, it appears that neither has been studied before.
We determine the order of growth of their (fractional) chromatic numbers:

\begin{Thm}
\label{thm: kneser chromatic intro}
The fractional and ordinary chromatic numbers of $\cg(n)$ and $\kg(n)$ are all $\Theta(n \log n).$
\end{Thm}

(Again, see \S\ref{subsec: notation} for the notation $\Omega$, $O$, and $\Theta$.)
Moreover, the implied constants in the bounds on $\chi(\kg(n))$ are within a factor of $\ln(2)$.
The determination of the chromatic number of the ordinary Kneser graph $\kg(n,k)$ was the content of a famous conjecture due to Kneser \cite{Kneseroriginal} and settled in a celebrated theorem of Lov\'asz \cite{LovaszKneser}.
Kneser exhibited a proper coloring of $\kg(n,k)$ using $n-2k+2$ colors, and Lov\'asz proved its optimality by defining the neighborhood complex $N(G)$ of a graph $G$, showing that the connectivity of $N(G)$ bounds $\chi(G)-3$ from below, and applying this bound to $\kg(n,k)$.
By contrast, the chief difficulty in Theorem \ref{thm: kneser chromatic intro} lies in establishing the upper bound on $\chi(\kg(n))$, which we accomplish by a variation on Kneser's original coloring of $\kg(n,k)$.

In \S \ref{sec: planar surfaces}, we formalize the argument from the Prelude and apply Theorem \ref{thm: kneser chromatic intro} in order to determine the order of growth of the chromatic number of the curve graph of the $n$-holed sphere $\Sigma_n$:

\begin{Thm}
\label{thm: planar chromatic intro}
The fractional and ordinary chromatic numbers of $\cc(\Sigma_n)$ are $\Theta(n\log n)$.
\end{Thm}

Prior to our work, Radhika Gupta obtained the estimate $\chi(\cc(\Sigma_n)) =O(n^2)$ \cite{Radhika}.

In \S \ref{sec: separating curves}, we generalize Theorem \ref{thm: planar chromatic intro} to
an estimate on the chromatic number of $\cs(S)$, the subgraph of $\cc(S)$ induced on the separating curves of an arbitrary compact surface $S$:

\begin{Thm}
\label{thm:separating}
If $S$ has Euler characteristic $-k < 0$, then $\chi(\cs(S)) = \Theta\left( k \log k \right).$
\end{Thm}

As with the case of a planar surface, the bounds come from embedding a cyclic interval graph into $\cs(S)$ and mapping it to a Kneser graph.
However, the homomorphism to the Kneser graph is subtler in this more general setting.
Drawing inspiration from \cite{BuserDistribution}, we place a hyperbolic metric on $S$ compatible with a pants decomposition in which all of the pant cuffs are very short.
The simple closed geodesics on this hyperbolic surface congregate near a 1-complex whose complement consists of $12k$ regions.
A separating simple closed geodesic partitions these regions into two parts, and disjoint geodesics yield distinct nested partitions.
In this way, we obtain the required homomorphism to the Kneser graph $\kg(12k)$.
Theorem \ref{thm: hyperbolic} contains the precise statement that we require, and we give a careful argument through a sequence of lemmas in hyperbolic geometry.
Ian Biringer has suggested an alternate description of this homomorphism based instead on train tracks, thereby eliminating the need for hyperbolic geometry in proving Theorem \ref{thm:separating}.

In \S \ref{sec: homologous}, we study the subgraph $\cc_v(S)$ of $\cc(S)$ induced on the curves that represent a fixed non-zero homology class $v \in H_1(S;\bz)$.
In this case, we obtain not only an exact answer but a uniqueness result:

\begin{Thm}
\label{thm: cc_v intro}
For any $v\ne0$, $\cc_v(S)$ is uniquely $t$-colorable, where $t$ denotes the clique number of $\cc_v(S)$.
\end{Thm}

See Theorems \ref{thm: cc_v(S_g)}, \ref{thm: C_v uniquely colorable}, and \ref{thm: non-nullhomologous} for more precise statements.
The proof of Theorem \ref{thm: cc_v intro} relies on a different set of techniques than those appearing up to this point.
The color of a curve in $\cc_v(S)$ is based on the genus $h$ of an immersed subsurface that it cobounds with a fixed reference curve in its homology class: for instance, when $S$ is closed and has genus $g$, the color is the value $h \pmod{g-1}$.
In fact, the coloring coincides with the {\em signed length} introduced by Irmer \cite[\S 4.1]{Irmerchillingworth}.
As a byproduct, the coloring permits an interpretation of signed length in terms of genera of immersed surfaces in the case of a closed surface; moreover, the use of domains in \S \ref{subsec: domains} makes it easy to calculate.
To prove that the coloring of $\cc_v(S)$ is unique, we show that it is possible to connect any two maximal cliques through a sequence in which each consecutive pair meet in all but one vertex: in other words, the maximal clique graph $\ck(\cc_v(S))$ is connected.
We do so by studying the action of a relevant mapping class group on $\ck(\cc_v(S))$ and applying Putman's trick \cite{Putmanconnectivity}. 

On combination of Theorems \ref{thm:separating} and \ref{thm: cc_v intro}, we bound the chromatic number of the curve graph of a closed surface as follows:

\begin{Thm}
\label{thm: s_g intro}
$g \cdot \log g \le \chi(\cc(S_g)) \le g \cdot 4^g$.
\end{Thm}

The lower bound comes simply from separating curves.
The upper bound comes from partitioning curves according to their$\pmod 2$ homology classes and coloring the curves in each class separately.
The zero class consists of the separating curves, and each of the $4^g-1$ non-zero classes $\overline{v}$ are disjoint unions of the graphs $\cc_v(S_g)$, where $v$ reduces to $\overline{v} \pmod 2$.
Theorem \ref{thm: cc_v intro} and a small maneuver around the axiom of countable choice (Proposition \ref{prop: (mod m)}) lets us color all the curves representing $\overline{v}$ by $g-1$ colors.
The upper bound in Theorem \ref{thm: s_g intro} then follows.

Thus, the chromatic number of $\cc(S_g)$ is super-linear and at most exponential in $g$.
These bounds leave great room for improvement, and we suspect the truth lies closer to the lower bound.
On the other hand, a construction of 1-systems on $S_g$ of size larger than $g^2 \cdot \log g$ would lead to an improvement on our lower bound.

In \S \ref{sec: permuting colors}, we explore a consequence of Theorem \ref{thm: cc_v intro} for a closed surface $S = S_g$.
Set $H = H_1(S,\bz)$ and let $\ci < \Mod(S)$ denote the \emph{Torelli group}, the mapping classes that act trivially on $H$.
For each primitive non-zero class $v \in H$, the group $\ci$ acts by automorphisms on $\cc_v(S)$.
By uniqueness of the minimal coloring, it permutes the color classes in a minimal coloring. 
In fact, it permutes them cyclically, and the action determines a homomorphism $\chi:\ci \to \Hom(H,\bz/(g-1)\bz)$ (Lemma \ref{lemma: chi homo}).
We explicitly compute the action of generators for $\ci$ on the color classes, and we obtain a relationship between the coloring of $\cc_v(S)$ and the \emph{Johnson homomorphism}.
The latter is a homomorphism $\tau:\ci \to \Hom(H,\bigwedge^2 H)/H$ that captures the free part of $H_1(\ci;\bz)$.
Johnson showed that the composition of $\tau$ with the algebraic intersection pairing $\bigwedge^2 H \to \bz$ and reduction$\pmod{g-1}$ gives the \emph{Chillingworth homomorphism} $t:\ci \to \Hom(H,\bz/(g-1)\bz)$ \cite{ChillingworthwindingI,ChillingworthwindingII,Johnsonabelian}. We show:

\begin{Thm}
\label{thm: coloring permutation}
The color permutation homomorphism $\chi$ equals the Chillingworth homomorphism $t$.
\end{Thm}

In fact, the identification of the coloring of $\cc_v(S)$ with Irmer's signed length immediately implies that the coloring homomorphism $\chi$ is equal to her \emph{stable length homomorphism} $\phi$ \cite[Lemma 5]{Irmerchillingworth}.
Irmer proves Theorem \ref{thm: coloring permutation} with $\phi$ in place of $\chi$ \cite[Thm.~1]{Irmerchillingworth}, and our proof mimics hers. 
Thus, the novelty in Theorem \ref{thm: coloring permutation} compared to Irmer's work is its interpretation in terms of coloring.
As a result, in Corollary \ref{cor: re-coloring} we recast the coloring of $\cc_v(S)$ in terms of $t$.

In \S \ref{sec: arc graphs}, we explore analogues of the preceding results for arc graphs.
The vertices of the arc graph $\ca(S)$ are isotopy classes of essential properly embedded arcs on $S$,  and two classes are adjacent if they have disjoint representatives.
Here the isotopy is free on the boundary.
We specialize to the case of a planar surface.
We show that the chromatic number of the subgraph $\ca \cs(\Sigma_n)$ induced on the separating arcs grows like $n \log n$ (Theorem \ref{thm: separating arcs}) and the subgraph of $\ca(\Sigma_n)$ induced on arcs representing a fixed non-zero homology class is uniquely $(n-2)$-colorable (Theorem \ref{thm: arcs unique}).
The analogue to Theorem \ref{thm: s_g intro} in this setting is:

\begin{Thm}
\label{thm: planar arc graph intro}
$\Omega(n \log n) = \chi(\ca(\Sigma_n)) = O(n^3)$.
\end{Thm}

The fact that the upper bound is polynomial and not exponential derives from the fact that the number  of$\pmod 2$ homology classes of arcs on $\Sigma_n$ is quadratic in $n$, whereas the number of$\pmod 2$ homology classes of curves on $S_g$ is exponential in $g$.
Doubling $\Sigma_n$ along its boundary induces an inclusion $\ca(\Sigma_n) \into \cc(S_{n-1})$.
Thus, an improved lower bound on $\chi(\ca(\Sigma_n))$ would result in a corresponding improvement on $\chi(\cc(S_g))$, and conversely for the upper bounds.
It seems likely, albeit less direct, that an improvement on the lower bound on $\chi(\cc(S_g))$ would inform one on $\chi(\ca(\Sigma_n))$, and conversely for the upper bounds.

In \S \ref{sec: four-holed} and \S\ref{sec: genus 2}, we obtain exact results for surfaces of low complexity.
In \S \ref{sec: four-holed}, we study $\ca(\Sigma_4)$ and various subgraphs of it.
For example, we show that the subgraph of $\ca(\Sigma_4)$ induced on arcs with exactly one endpoint on a fixed boundary component has chromatic number 4, with color classes corresponding to the orbits under the action of the level-3 congruence subgroup $\Gamma(3) < \mathrm{PSL}(2,\bz)$ on $P^1(\bz^2)$ (Theorem \ref{thm: chi(a)=4}).
In \S\ref{sec: genus 2}, we study $\cc(S_2)$ and its subgraph $\cn(S_2)$ induced on the nonseparating curves.
Using the hyperelliptic involution, we show:

\begin{Thm}
\label{thm: S_2}
$\chi(\cn(S_2)) = 4$ and $\chi(\cc(S_2))=5$.
\end{Thm}

Finally, in \S \ref{sec: conclusion}, we collect some questions for further study.


\subsection{Conventions and Notation}
\label{subsec: notation}
If $f$ and $g$ denote two real-valued functions, then we write $f = O(g)$ if there exists an absolute constant $C > 0$ such that $f \le C \cdot g$.
We write $f = \Omega(g)$ if $g = O(f)$, and we write $f = \Theta(g)$ if $f = \Omega(g)$ and $f = O(g)$.

All surfaces appearing in our results are compact, connected, and orientable.
We denote by $S_g^b$ a surface of genus $g$ with $b$ boundary components, or {\em holes}.  If $b=0$, then we suppress it from the notation, and if $g=0$, then we write $\Sigma_n = S^n_0$.

The \emph{mapping class group} of $S$ is the group $\Mod(S) := \pi_0 ( \mathrm{Homeo}^+(S))$. 
For convenience, we break slightly with the convention of \cite{Farbprimer} by allowing $S$ to permute boundary components; the subgroup acting trivially on boundary components is denoted by $\PMod(S)$. 
This difference is relevant in the proofs of Proposition \ref{prop: connected} and Theorem \ref{thm: arcs unique}.

A simple closed curve on a surface is \emph{separating} if its complement is disconnected; it is \emph{peripheral} if it is isotopic to a hole; and it is \emph{essential} if it neither is peripheral nor bounds a disk.
Similarly, a homology class is \emph{separating} or \emph{peripheral} if it is represented by an oriented simple closed curve with the corresponding property.

We denote the curve (resp.~arc) graph by $\cc(S)$ (resp.~$\ca(S)$), and the subgraphs induced by separating and non-separating curves (resp.~arcs) by $\cs(S)$ and $\cn(S)$ (resp.~$\ca\cs(S)$ and $\ca\cn(S)$). 
Given a homology class $v$ (resp.~relative to $\partial S$), we denote the subgraph spanned by curves (resp.~arcs) which can be oriented to be homologous to $v$ by $\cc_v(S)$ (resp.~$\ca_v(S)$).

When convenient, we elide the difference between a graph and the flag simplicial complex with the same 1-skeleton (in which any complete subgraph on $k$ vertices spans a unique $(k-1)$-simplex). 
For example, $\cc(S)$ may also denote the curve \emph{complex} associated to $S$.

Given a simplicial complex $\cc$, we let $\ck(\cc)$ denote the graph whose vertices consist of the maximal simplices in $\cc$, where two maximal simplices are adjacent if they meet in a codimension-1 face. We refer to $\ck(\cc)$ as the \emph{maximal clique graph} of $\cc$. The reader is cautioned that $\ck(\cc(S_g))$ is not quite the `clique graph' in the sense of \cite[p.~3]{Kim-Koberda}, in which vertices correspond to (not necessarily maximal) cliques.


\section*{Acknowledgments}
We thank Ian Biringer and Peter Feller for pleasant conversations throughout the course this work.
We especially thank Ian for explaining how to use train tracks in place of hyperbolic geometry in the proof of Theorem \ref{thm:separating}.
We also thank Ken Bromberg for helpful email correspondence.
NGV is grateful to Peter Heinig for bringing the chromatic number of the curve graph to his attention.
JEG was supported by NSF CAREER Award DMS-1455132 and an Alfred P. Sloan Foundation Research Fellowship.
NGV was supported in part by NSF RTG grant 1045119.


\section{Kneser graphs and cyclic interval graphs}
\label{sec: total kneser}

In this section, we introduce the total Kneser graph $\kg(n)$ and its subgraph the total cyclic interval graph $\cg(n)$.
These graphs are natural unions of the well-known Kneser graphs $\kg(n,k)$ and cyclic interval graphs $\cg(n,k)$.
We determine the fractional chromatic numbers of these graphs and the growth orders of their chromatic numbers in Theorems \ref{thm: kneser fractional chromatic} and \ref{thm: kneser chromatic}.
For background on graph theory, including discussion about homomorphisms, (fractional) chromatic numbers, and the two-parameter Kneser and cyclic interval graphs, see \cite[Ch.7]{Godsilgraphs}.

Given a pair of positive integers $n \ge 2k$, the {\em Kneser graph} $\kg(n,k)$ is the graph whose vertices are the $k$-element subsets of $\{1, \ldots, n\}$ and whose edges are unordered pairs of disjoint subsets.
A {\em cyclic interval} is a cyclic shift of the set $\{1,\dots,k\}$ modulo $n$.
The {\em cyclic interval graph} $\cg(n,k)$ is the subgraph of $\kg(n,k)$ induced on the cyclic intervals.

The fractional chromatic numbers of these graphs are well-known \cite[\S7.7]{Godsilgraphs}, 
and the determination of the chromatic number of the Kneser graph $\kg(n,k)$ is a celebrated theorem of Lov\'asz \cite{LovaszKneser}.
We record these values here:

\begin{Thm}
\label{thm: original kneser chromatic}
$\chi_f(\cg(n,k)) = \chi_f(\kg(n,k)) = n/k$, $\chi(\cg(n,k)) = \lceil n/k \rceil$, and $\chi(\kg(n,k)) = n-2k+2$. \qed
\end{Thm}

Given a positive integer $n \ge 2$, the {\em total Kneser graph} $\kg(n)$ is the graph whose vertices are partitions of $\{1,\dots,n\}$ into an unordered pair of non-empty disjoint subsets $(A,B)$.
We often express a partition just by one of its parts, since there is no loss of information.
Two such partitions $(A,B)$, $(C,D)$ are {\em nested} if one of $A$ or $B$ is contained in one of $C$ 
or $D$; note that the condition is symmetric in the two pairs.
The edges of $\kg(n)$ are pairs of distinct nested partitions.
The {\em total cyclic interval graph} $\cg(n)$ is the subgraph of $\kg(n)$ induced on partitions in which the parts are cyclic intervals.
Observe that for $k < \frac n2$, $\kg(n,k) \subset \kg(n)$ is induced on the partitions $(A,B)$ with $\min \{ |A|, |B| \} = k$, and $\cg(n,k) = \cg(n) \cap \kg(n,k)$.  When $n=2k$, a mild discrepancy arises, but it does not influence our results.

A cyclic interval contains a {\em minimal element}, the shift of 1.
We label a vertex $(A,B) \in \cg(n)$ by the pair $(i,j)$ that records the minimal elements of $A$ and $B$, with the convention that $i < j$.
Two such pairs $(i,j)$ and $(i',j')$ are {\em linked} if $i < i' < j < j'$ or $i' < i < j' < j$.
Under the labeling by pairs, edges in $\cg(n)$ correspond precisely to unlinked pairs.

\begin{Lem}
\label{l: indep set bound}
If $S$ is an independent set in $\cg(n)$, then $|S| \le \min\{|A| \, | \, (A,B) \in S\}.$
\end{Lem}

\begin{proof}
Choose any $(A,B) \in S$.
By applying a cyclic shift to $\cg(n)$, we may assume that $j-i = |A|$, where $(A,B)$ gets labeled $(i,j)$.
Since $S$ is an independent set, the labels in $S$ are pairwise linked.
Select any other label $(i',j')$ in $S$.
Note that exactly one of $i'$ and $j'$ lies between $i$ and $j$.
Furthermore, no two labels in $S$ share a common coordinate.
It follows that there are at most $j-i-1$ labels in $S$ that link with $(i,j)$, so $|S| \le j-i = |A|$, as desired.
\end{proof}

For a positive integer $m$, let $H_m$ denote the $m$-th harmonic number $\sum_{k=1}^m \frac 1k$,
and let $p(m) \in \{0,1\}$ denote the least non-negative residue of $m\pmod 2$.
Recall the bounds $\log(m) < H_m \le \log(m)+1$ for the following result.

\begin{Thm}
\label{thm: kneser fractional chromatic}
 $\chi_f(\kg(n)) = \chi_f(\cg(n)) = n \cdot H_{\lfloor (n-1)/2 \rfloor} + (1-p(n)) = \Theta(n \log(n))$.
\end{Thm}

\begin{proof}
The fractional chromatic number of $\kg(n)$ is bounded above by the sum of the fractional chromatic numbers of its vertex-disjoint induced subgraphs $\kg(n,k)$, $1 \le k < n/2$, and, if $n$ is even, the independent set induced on the partitions into $\frac n2$-subsets.
Since $\chi_f(\kg(n,k)) = n/k$ by Theorem \ref{thm: original kneser chromatic}, we obtain the required upper bound.

Next, define $w: V(\cg(n)) \to \br_{\ge 0}$ by $w((A,B)) = 1 / |A|$, where $|A| \le |B|$.
Let $S$ be an independent set in $\cg(n)$.
Select $(A,B) \in S$ with $|A|$ minimal.
We have $w(S) \le |S| \cdot w(A) \le |A| / |A| = 1$, using Lemma \ref{l: indep set bound} in the second inequality.
Therefore, $w$ is a fractional clique.
Its total value equals the required lower bound.
\end{proof}

\begin{Rem}
The fractional clique defined in the proof of Theorem \ref{thm: kneser fractional chromatic} is simply the sum of the optimal fractional cliques for the subgraphs $\cg(n,k)$ that get used to establish their fractional chromatic numbers (Theorem \ref{thm: original kneser chromatic}).
\end{Rem}

The union bound  leads to the soft estimate
\[
\chi(\kg(n)) \le \sum_{k=1}^{\lfloor n/2 \rfloor} \chi(\kg(n,k)) = \sum_{k=1}^{\lfloor n/2 \rfloor} (n-2k+2) = \Theta(n^2),
\]
which grows faster than $\chi_f(\kg(n))$.
The following Theorem uses a refined coloring to show that $\chi_f(\kg(n))$ and $\chi(\kg(n))$ grow at the same rate, and in fact differ by a factor of no more than $\ln(2) \approx 0.69$.
It is a variation on Kneser's original $(n-2k+2)$-coloring of $\kg(n,k)$.

\begin{figure}
\includegraphics[width=3.5in]{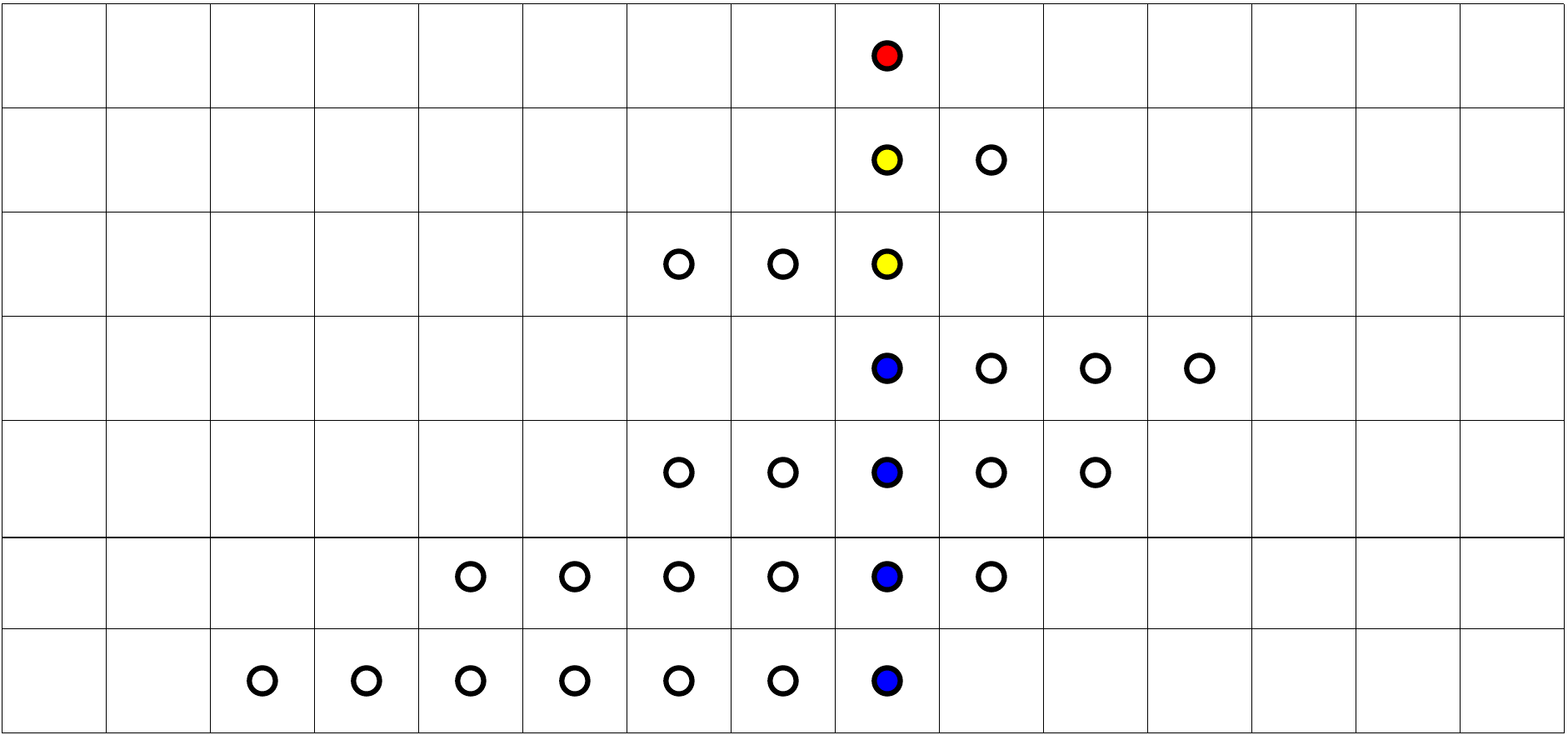}
\put(-246,123){1}
\put(-230,123){2}
\put(-214,123){3}
\put(-197,123){4}
\put(-180,123){5}
\put(-163,123){6}
\put(-146,123){7}
\put(-129,123){8}
\put(-112,123){9}
\put(-98,123){10}
\put(-81,123){11}
\put(-64,123){12}
\put(-47,123){13}
\put(-30,123){14}
\put(-13,123){15}
\put(-268,107){$A_1$}
\put(-268,90){$A_2$}
\put(-268,73){$A_3$}
\put(-268,56){$A_4$}
\put(-268,39){$A_5$}
\put(-268,22){$A_6$}
\put(-268,5){$A_7$}
\caption{Vertices $A_i \in \kg(15)$, $|A_i|=i$, $i=1,\dots,7$. $A_1$ gets color $(0,9)$; $A_2, A_3$ get color $(1,9)$; and $A_4,A_5,A_6,A_7$ get color $(2,9)$.
No two vertices of the same color are disjoint or nested.
}
\label{fig: kneser}
\end{figure}

\begin{Thm}
\label{thm: kneser chromatic}
$\chi(\kg(n)) \leq n\cdot \left\lceil \log_2\left(\frac n2\right) \right\rceil +1= O(n \log(n)).$
\end{Thm}

\begin{proof}
Select a subset $A \subset \{1, \ldots, n\}$ with $|A| \leq \frac n2$.  We can uniquely express $|A| = 2^{k+1}-l$, where $k,l \in \bz, k \geq 0$, and $1\leq l \leq 2^k$.  Let $a$ denote the $l$-th largest element of $A$.  Assign $A$ the color consisting of the pair $(k,a)$.  See Figure \ref{fig: kneser}.  Observe that this coloring uses $n\cdot \left\lceil \log_2\left(\frac n2\right) \right\rceil$ colors, unless $n$ is a power of 2.  If it is, then we alter the coloring on the $\frac n2$-element subsets by giving them all the same color, distinct from those used on the subsets with fewer elements.  In this case, the coloring uses $n\cdot \left\lceil \log_2\left(\frac n2\right) \right\rceil +1$ colors. We claim in either case that this coloring of $\kg(n)$ is proper.
Suppose that $A$ and $B$ are different subsets that receive the same color $(k,a)$.  We seek to show that $A$ and $B$ are neither disjoint nor is one contained in the other.  Assume without loss of generality that $|A| \leq |B|$.
As $a \in A \cap B$, these subsets are not disjoint.  If $|A|=|B|$, then it is immediate that neither is contained in the other.  If instead $|A| < |B|$, then $l(A) > l(B)$, so $A$ contains more elements less than $a$ than $B$ does, while $2^{k+1}-2l(B) > 2^{k+1}-2l(A)$, so $B$ contains more elements greater than $a$ than $A$ does.  Therefore, neither of $A$ and $B$ is contained in the other in this case either.  It follows that this coloring of $\kg(n)$ is proper and establishes the desired bound.
\end{proof}

\begin{proof}
[Proof of Theorem \ref{thm: kneser chromatic intro}]
Immediate from Theorems \ref{thm: kneser fractional chromatic} and \ref{thm: kneser chromatic}.
\end{proof}


\section{Curves on planar surfaces}
\label{sec: planar surfaces}

In this section, we obtain a precise estimate on the chromatic number of the curve graph of a planar surface, and we determine its fractional chromatic number exactly.
The methods of this section serve as a prototype for those appearing later on.

We assume that $n \ge 5$, so that $\cc(\Sigma_n)$ contains edges.
(When $n=4$, the definition of $\cc(\Sigma_4)$ is usually altered so that edges consist of pairs of curves with minimal intersection number 2; see \S \ref{sec: four-holed}.)
We establish the following more precise version of Theorem \ref{thm: planar chromatic intro}:

\begin{Thm}
\label{thm:planar}
$\chi(\cc(\Sigma_n)) = \Theta(n \log n)$ and $\chi_f(\cc(\Sigma_n)) = \chi_f(\kg(n))-n$. 
\end{Thm}

We split the proof of Theorem \ref{thm:planar} into two easy lemmas.

\begin{Lem}
\label{lem:planar-upper}
There exists a homomorphism $f: \cc(\Sigma_n) \to \kg(n) \smallsetminus \kg(n,1)$.
\end{Lem}

\begin{proof}
A curve $c \in \cc(\Sigma_n)$ induces a partition $f(c)$ of the components of $\del \Sigma_n$ into two non-empty subsets according to which component of $\Sigma_n \smallsetminus c$ they belong.
Both subsets have size at least two, since $c$ is essential.
If $c,d \in \cc(\Sigma_n)$ are adjacent, then $\Sigma_n \smallsetminus c \cup d$ consists of three components, each of which contains a component of $\del \Sigma_n$.
It follows that the partitions $f(c)$ and $f(d)$ are distinct and nested.
Identifying the holes of $\Sigma_n$ with the underlying set of $\kg(n)$, the mapping $f$ defines the required homomorphism.
\end{proof}

\begin{Lem}
\label{lem:planar-lower}
There exists an embedding $c: \cg(n) \smallsetminus \cg(n,1) \hookrightarrow \cc(\Sigma_n)$.
\end{Lem}

\begin{proof}
Embed an $n$-cycle in the 2-sphere $\Sigma$.
Label its vertices $p_1,\dots,p_n$ cyclically and its edges $e_i=(p_i,p_{i+1})$, indices $(\mathrm{mod} \, n)$.
Identify $\Sigma_n$ with the complement in $\Sigma$ of a small neighborhood of $\{p_1,\dots,p_n\}$. 
For a vertex $v = \{i, i+1, \ldots, i+k\} \in \cg(n) \smallsetminus \cg(n,1)$, let $c(v) \in \cc(\Sigma_n)$  be the boundary of a regular neighborhood of $e_i \cup \cdots \cup e_{i+k-1} \subset \Sigma$. 
It is easy to see that the mapping $c$ defines the required embedding.  
\end{proof}

\begin{proof}[Proof of Theorem \ref{thm:planar}]
The result follows from Theorem \ref{thm: kneser chromatic intro}, Lemma \ref{lem:planar-upper}, Lemma \ref{lem:planar-lower}, and the monotonicity of the (fractional) chromatic number under homomorphisms.
Note in addition that $\kg(n,1) \subset \kg(n)$ and $\cg(n,1) \subset \cg(n)$ are cliques of size $n$ adjacent to all other vertices; removing them from their supergraphs lowers the (fractional) chromatic numbers by $n$.
\end{proof}

The proof of Theorem \ref{thm:planar} raises the question whether $\kg(n) \smallsetminus \kg(n,1)$ itself embeds in $\cc(\Sigma_n)$.
The following result indicates that this is not the case:

\begin{Prop}
\label{p: no petersen}
$\cc(\Sigma_5)$ does not contain a subgraph isomorphic to $\kg(5,2)$.
\end{Prop}

\begin{proof}
Suppose by way of contradiction that there were.
Restricting $f$ to the subgraph gives an endomorphism of $\kg(5,2)$. 
Since any endomorphism of $\kg(n,k)$ is an automorphism \cite[Theorem 7.9.1]{Godsilgraphs}, it follows that the subgraph is the image of a section $s$ of $f$.
Given a 2-element subset $\{i,j\}$, let $a_{ij}$ denote the arc, unique up to isotopy, with endpoints on the holes $\del_i$, $\del_j$ that is disjoint from $s(\{i,j\})$.
Collapsing each hole to a point gives a drawing of $K_5$ on the sphere such that arcs with distinct endpoints are disjoint.
If a pair of arcs with a common endpoint meet in their interiors, then exchanging portions of these arcs and performing a small isotopy results in such a drawing of $K_5$ with fewer self-intersections.
Iterating this process ultimately results in a planar drawing of $K_5$, a contradiction.
\end{proof}


\section{Separating curves} 
\label{sec: separating curves}

The goal of this section is to prove Theorem \ref{thm:separating}.
The proof strategy is similar to that of Theorem \ref{thm:planar}.  However, we must replace the holes by a less obvious collection of points.  They are provided by the following result, which draws inspiration from \cite{BuserDistribution}:

\begin{Thm}
\label{thm: hyperbolic}
If $S$ has Euler characteristic $-k < 0$, then there exists a hyperbolic metric on $S$ and a subset $Q \subset S$ of $12k$ points with the following two properties: \begin{inparaenum} \item $S$ has totally geodesic boundary, and \item if $F \subset S$ is a subsurface with totally geodesic boundary, then $Q \cap \del F = \emptyset$ and
$$
-\chi(F) = \frac 1{2 \pi} \mathrm{Area}(F) = \frac 1{12} |Q\cap F|.
$$
\end{inparaenum}
\end{Thm}

Thus, the discrete uniform measure concentrated on the point set $Q$ is proportional to the standard area measure when restricted to subsurfaces with totally geodesic boundary.

We first derive Theorem \ref{thm:separating} from Theorem \ref{thm: hyperbolic}.

\begin{proof}[Proof of Theorem \ref{thm:separating}]
First, we establish the upper bound. 
Apply Theorem \ref{thm: hyperbolic}.
A curve $c \in \cs(S)$ has a unique geodesic representative, which cuts $S$ into a pair of subsurfaces with totally geodesic boundary.
We obtain a partition $f(c)$ of $Q$ into two parts according to the subsurfaces these points lie in.
Identifying $Q$ with the underlying set of the Kneser graph $\kg(12k)$, we obtain a map $f: \cs(S) \to \kg(12 k)$.

We claim that $f$ is a homomorphism.
If $c, d \in \cs(S)$ are adjacent, then they can be realized by disjoint, simple closed geodesics in $S$. 
The complement $S \smallsetminus c \cup d$ consists of three components with totally geodesic boundary.
Since the curves are essential and not parallel, each component has negative Euler characteristic.
By Theorem \ref{thm: hyperbolic}, each subsurface contains points of $Q$, and it follows that $f(c)$ and $f(d)$ are nested.
Therefore, $f$ is a homomorphism, as claimed.
Theorem \ref{thm: kneser chromatic} and the monotonicity of the chromatic number under homomorphisms now lead to the stated upper bound.

For the lower bound, embed a planar surface $\Sigma = \Sigma_{g+b}$ into $S$ so that each component of $\del \Sigma$ either bounds a subsurface $S^1_1 \subset S$ or is a component of $\del S$. 
Every essential curve in $\Sigma$ is essential and separating in $S$, and distinct curves in $\Sigma$ are distinct in $S$. 
The embedding $\Sigma \into S$ therefore induces an embedding $\cc(\Sigma) \into \cs(S)$, and Theorem \ref{thm:planar} gives the desired lower bound.  
\end{proof}

We now develop the proof of Theorem \ref{thm: hyperbolic} through a sequence of lemmas in hyperbolic geometry.
Roughly speaking, we apply a geometric limiting argument in which the curves in a pants decomposition get pinched.
For background on hyperbolic geometry, see \cite[Ch.1\&3]{BuserGeometry} and \cite[\S\S10.5-6]{Farbprimer}.

\begin{figure}
\includegraphics[width=5in]{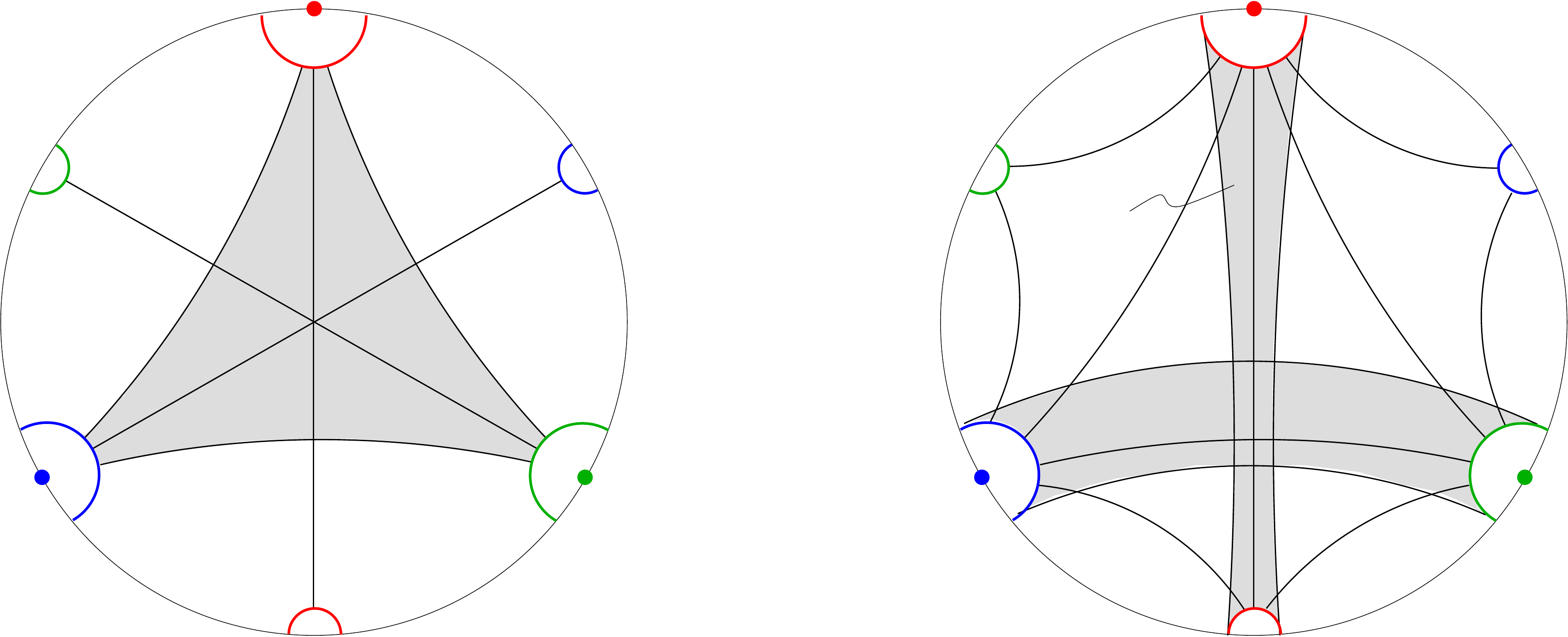}
\put(-390,150){(a)}
\put(-285,90){$H$}
\put(-300,150){\textcolor{red}{$p_1$}}
\put(-280,150){\textcolor{red}{$\del_1$}}
\put(-300,-15){\textcolor{red}{$\del_1'$}}
\put(-360,27){\textcolor{blue}{$p_2$}}
\put(-370,45){\textcolor{blue}{$\del_2$}}
\put(-220,100){\textcolor{blue}{$\del_2'$}}
\put(-220,35){\textcolor{green3}{$p_3$}}
\put(-230,20){\textcolor{green3}{$\del_3$}}
\put(-360,120){\textcolor{green3}{$\del_3'$}}
\put(-305,15){$g_{11}$}
\put(-350,60){$g_{12}$}
\put(-240,90){$g_{22}$}
\put(-270,35){$g_{23}$}
\put(-340,105){$g_{33}$}
\put(-275,120){$g_{13}$}
\put(-170,150){(b)}
\put(-115,90){$Q_{11}$}
\put(-60,50){$Q_{23}$}
\caption{(a) A right-angled hexagon $H$ and distinguished geodesics in $\bd^2$. (b) Two ideal quadrilaterals and three reflections of $H$.}
\label{fig: geometry}
\end{figure}
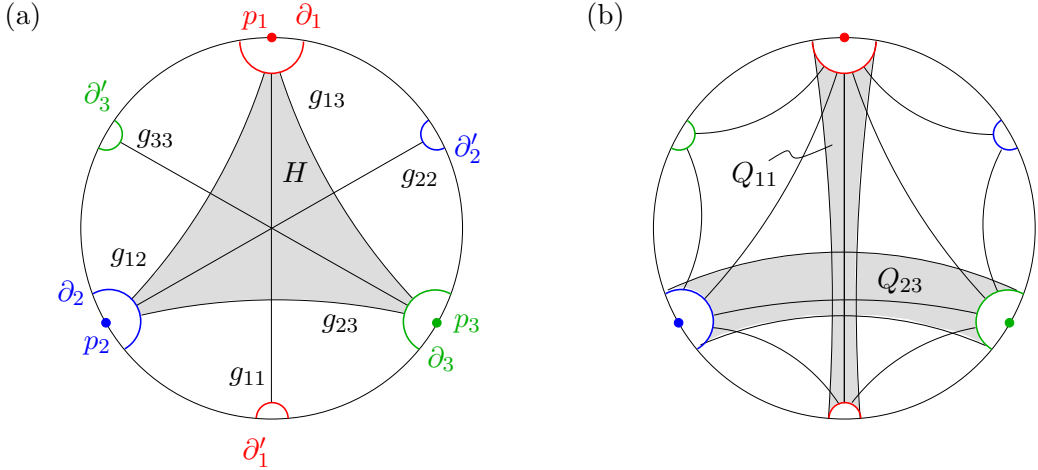

We work in the Poincar\'e disk model $\bd$ of the hyperbolic plane.
Fix three equally spaced ideal points $p_1,p_2,p_3 \in \del \overline \bd$.
Consider a right-angled geodesic hexagon $H \subset \bd$ that is invariant under the symmetries of $\bd$ permuting the points $p_i$.
Extend the side of $H$ closest to $p_i$ to a complete geodesic $\del_i$, and label its side with endpoints on $\del_i$ and $\del_j$ by $g_{ij}$.
Let $\del_i'$ denote the reflection of $\del_i$ across $g_{jk}$, where $\{i,j,k\} = \{1,2,3\}$, and let $g_{ii}$ denote the geodesic arc through the origin and perpendicular to $\del_i$ with endpoints on $\del_i$ and $\del_i'$.
See Figure \ref{fig: geometry}(a).

The reflections across the sides $g_{ij}$ of $H$ generate a subgroup $\Gamma < \mathrm{Isom}(\bd)$, and the images of $H$ under the action by $\Gamma$ tesselate a simply-connected, convex region $\tilde P \subset \bd$.
For $1 \le i \le j \le 3$, let $Q_{ij} \subset \tilde P$ denote the interior of the ideal quadrilateral determined by the boundary components of $\tilde P$ containing the endpoints of $g_{ij}$.  See Figure \ref{fig: geometry}(b).

\begin{Lem}
\label{lem: tesselate}
The images of $H \cap Q_{ij}$ under the reflection group $\Gamma$ cover $Q_{ij}$.
\end{Lem}

\begin{proof}
Let $\sigma$ and $\sigma'$ denote the sides of $\overline{Q}_{ij}$ interior to $\tilde P$.
Moving along $\sigma$ from the interior of $H$ to the ideal point of intersection between $\sigma$ and $\del_i$, we encounter a sequence of edges $s_1,s_2,\dots$ in the tesselation of $\tilde P$.
The geodesic triangle $\Delta \subset Q_{ij}$ bounded by $\del_i$, $\sigma$, and $s_1$ is tesselated by quadrilaterals $Q_n$, $n \ge 1$, where $Q_n$ is bounded by $s_n,\sigma,s_{n+1}$, and $\del_i$.
Let $r_n$ denote the reflection in the side $s_n$.
For $n\geq 2$, observe that $r_n(Q_n)$ is the quadrilateral bounded by $s_n, r_n(\sigma), s_{n-1}$, and $\del_i$.  
The geodesic $r_n(\sigma)$ passes through the point of intersection between $\sigma$ and $s_n$ and limits to the ideal point of intersection between $\del_i$ and $\sigma'$, and the portion of $r_n(\sigma)$ between these two points is contained in $Q_{ij}$ by convexity.
It follows that $r_n(Q_n) \subset Q_{n-1}$ for all $n \ge 2$.
Furthermore, $r_1(Q_1)$ is contained in $Q_{ij} \cap H$ by similar reasoning.
It follows that the images of $Q_{ij} \cap H$ under the elements $r_n \cdots r_1 \in \Gamma$, $n \ge 1$, cover $\Delta$.
Let $k$ and $l$ denote the unique values so that $\{i,j\}$ and $\{k,l\}$ partition the set $\{1,2,3\}$.
The region $Q_{ij}$ is now the union of $Q_{ij} \cap H$ and the images of $\Delta$ under the four elements generated by the reflections across the orthogonal sides $g_{ij}$ and $g_{kl}$.
The statement of the Lemma now follows with the additional observation that $g_{ii}\Gamma g_{ii} = \Gamma$ for $i \in \{1,2,3\}$.
\end{proof}

Let $\ell > 0$ denote the common length of the sides of $H$ contained in the $\del_i$.
Observe that we obtain a one-parameter family of hexagons $H$ as above by varying the value $\ell \in \br^+$.

\begin{Lem}
\label{lem: limit}
As $\ell \to 0$, the complement of the $Q_{ij}$ in $H$ consists of 6 pairwise disjoint geodesic triangles whose angles limit to $(\pi/2,\pi/3,0)$.
\end{Lem}

\begin{proof}
The endpoints of $\del_i$ are continuous functions of $\ell$ and they converge to $p_i$ as $\ell \to 0$.
Therefore, as $\ell \to 0$, the $Q_{ij}$ are converging to complete geodesics $G_{ij}$ and $H$ is converging to the ideal triangle $\Delta$ with vertices $p_1, p_2, p_3$ (the convergence is in the Hausdorff metric with respect to the Euclidean metric on $\bd$).
Observe that the complement of the $G_{ij}$ in $\Delta$ is the union of six  $(\pi/2,\pi/3,0)$-triangles.
The lemma now follows immediately.
\end{proof}

The subgroup $\Gamma_0 = \Gamma \cap \mathrm{Isom}^+(\bd)$ has index two in $\Gamma$.
The quotient $P = \tilde P / \Gamma_0$ is homeomorphic to a pair of pants $S^3_0$, and the quotient map $p: \tilde P \to P$ is its universal covering.
The map $p$ endows $P$ with a hyperbolic structure in which each boundary component $p(\del_i)$ is totally geodesic and has length $2\ell$.

\begin{Lem}
\label{lem: geodesic cover}
Every properly embedded, simple geodesic arc in $P$ is contained in the image under $p$ of the $Q_{ij}$.
\end{Lem}

\begin{proof}
Let $a_{ij} \subset P$ denote a properly embedded, simple geodesic arc that has one endpoint on $p(\del_i)$ and the other on $p(\del_j)$.
The projection $p(g_{ij})$ is a properly embedded, simple geodesic arc isotopic to $a_{ij}$.
The isotopy from $p(g_{ij})$ to $a_{ij}$ lifts to an isotopy from $g_{ij}$ to a lift $\tilde a_{ij} \subset \tilde P$.
In particular, $\tilde a_{ij}$ is a geodesic arc with endpoints on the same boundary components of $\tilde P$ as $g_{ij}$.
It follows that $\tilde a_{ij} \subset Q_{ij}$.
Therefore, $a_{ij} = p(\tilde a_{ij}) \subset p(Q_{ij})$, as required.
\end{proof}

\begin{Prop}
\label{prop: simple geodesics}
As $\ell \to 0$, the complement of the properly embedded, simple geodesic arcs in $P$ contains 12 pairwise disjoint geodesic triangles whose angles limit to $(\pi/2,\pi/3,0)$.
\end{Prop}

\begin{proof}
Any reflection in $\Gamma \smallsetminus \Gamma_0$ descends to the same order-two, orientation-reversing isometry $r : P \to P$.
The covering map $p$ is a homeomorphism from $H$ onto its image, and $p(H)$ and $r(p(H))$ tesselate $P$.
The image of the $Q_{ij}$ under $p$ is the image of the $Q_{ij} \cap H$ under $p$ and its reflection under $r$, by Lemma \ref{lem: tesselate}.
The 12 triangles are then the images under $p$, and their reflections under $r$, of the 6 triangles in Lemma \ref{lem: limit}.
They lie in the complement of the properly embedded, simple geodesic arcs in $P$ by Lemma \ref{lem: geodesic cover}.
\end{proof}

We are now ready to prove Theorem \ref{thm: hyperbolic}.

\begin{proof}[Proof of Theorem \ref{thm: hyperbolic}]
Choose a pants decomposition $S = P_1 \cup \cdots \cup P_k$.
Fix $\ell > 0$ and place a hyperbolic structure on $S$ so that each boundary component of each $P_i$ has length $\ell$.
Every simple closed geodesic in $S$ intersects each $P_i$ in a union of disjoint simple proper arcs or in a boundary component.
By Proposition \ref{prop: simple geodesics}, as $\ell \to 0$, the set of simple closed geodesics in $S$ is disjoint from a set of $12k$ pairwise disjoint geodesic triangles $\Delta_1,\dots,\Delta_{12k}$, 12 in each $P_i$, with the area of each $\Delta_i$ limiting to $\pi/6$, the area of a $(\pi/2,\pi/3,0)$-triangle.
It follows in turn that the complement of the $\Delta_i$ has area limiting to 0 as $\mathrm{Area}(S) = 2\pi k$.
Thus, for $\ell$ chosen suitably small, the area of a subsurface $F \subset S$ with totally geodesic boundary must equal to within 0.5, say, of $\pi/6$ times the number of the $\Delta_i$ that it contains.
On the other hand, $\mathrm{Area}(F) = -2 \pi \chi(F)$ by the Gauss-Bonnet formula, which is an integral multiple of $\pi/6$.
It follows that $\mathrm{Area}(F)$ equals to {\em exactly} $\pi/6$ times the number of the $\Delta_i$ that it contains.
To finish the proof, let $Q$ consist of a single point from each $\Delta_i$. 
\end{proof}


\section{The graph of homologous curves}
\label{sec: homologous}

The goal of this section is to study the chromatic number of the subgraph of $\cc(S)$ induced on curves that can be oriented to represent a fixed non-zero homology class.
The main theme is that this subgraph has a unique minimal coloring, which can be expressed in terms of genera of immersed surfaces.
\S \ref{subsec: closed cc_v} treats the case of coloring a closed surface.
\S \ref{subsec: domains} recasts this coloring in terms of domains, which is critical for the  material of \S \ref{sec: permuting colors}.
\S \ref{subsec: upper bound} applies the coloring to establish the upper bound on $\chi(\cc(S_g))$ stated in Theorem \ref{thm: s_g intro}.
\S \ref{subsec: uniqueness} establishes the uniqueness of the minimal coloring in the case of a closed surface.
Finally, \S \ref{subsec: with boundary} treats the case of a surface with boundary.


\subsection{The chromatic number of $\cc_v(S_g)$}
\label{subsec: closed cc_v}
Let $v \in H_1(S_g;\bz)$ denote a non-zero primitive element, and let $\cc_v(S)$ denote the subgraph of $\cc(S)$ induced on the curves that can be oriented to represent $v$.
The main result of this subsection is:

\begin{Thm}
\label{thm: cc_v(S_g)}
$\chi(\cc_v(S_g)) = \omega(\cc_v(S_g)) = g-1$.
\end{Thm}

\begin{Rem} Although its clique number and chromatic number are equal, the graph $\cc_v(S_g)$ is typically not {\em perfect}: not all of its induced subgraphs have this property. Figure \ref{not perfect pic} displays an induced five-cycle in $\cc_v(S_g)$ for $g \ge 6$.
\end{Rem}

\begin{figure}
	\centering
	\begin{minipage}{.45\textwidth}
	\centering
	\includegraphics[width=6cm]{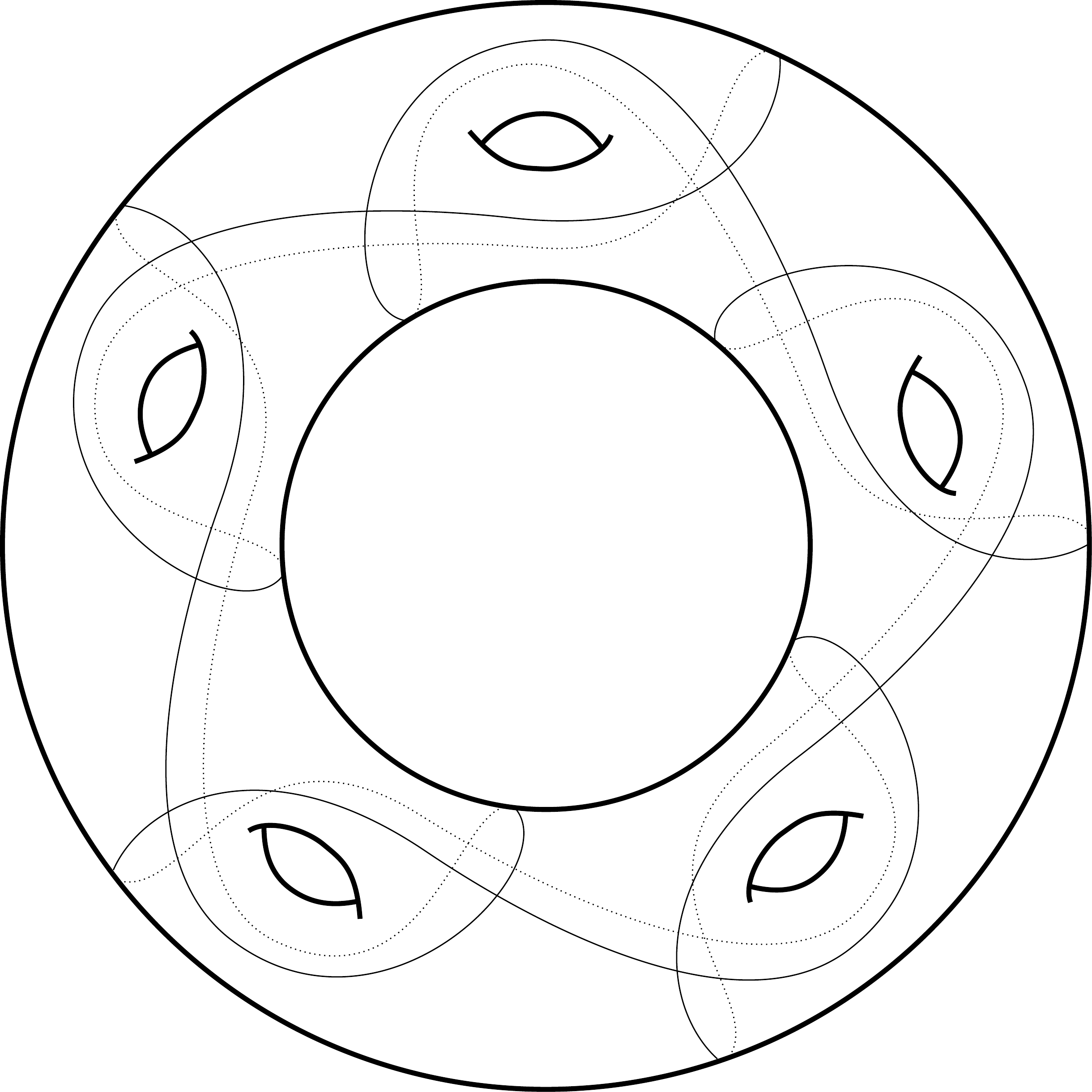}
	\subcaption{An induced five-cycle.}
	\label{not perfect pic}
	\end{minipage} \hfill
	\begin{minipage}{.45\textwidth}
	\centering
	\includegraphics[width=6cm]{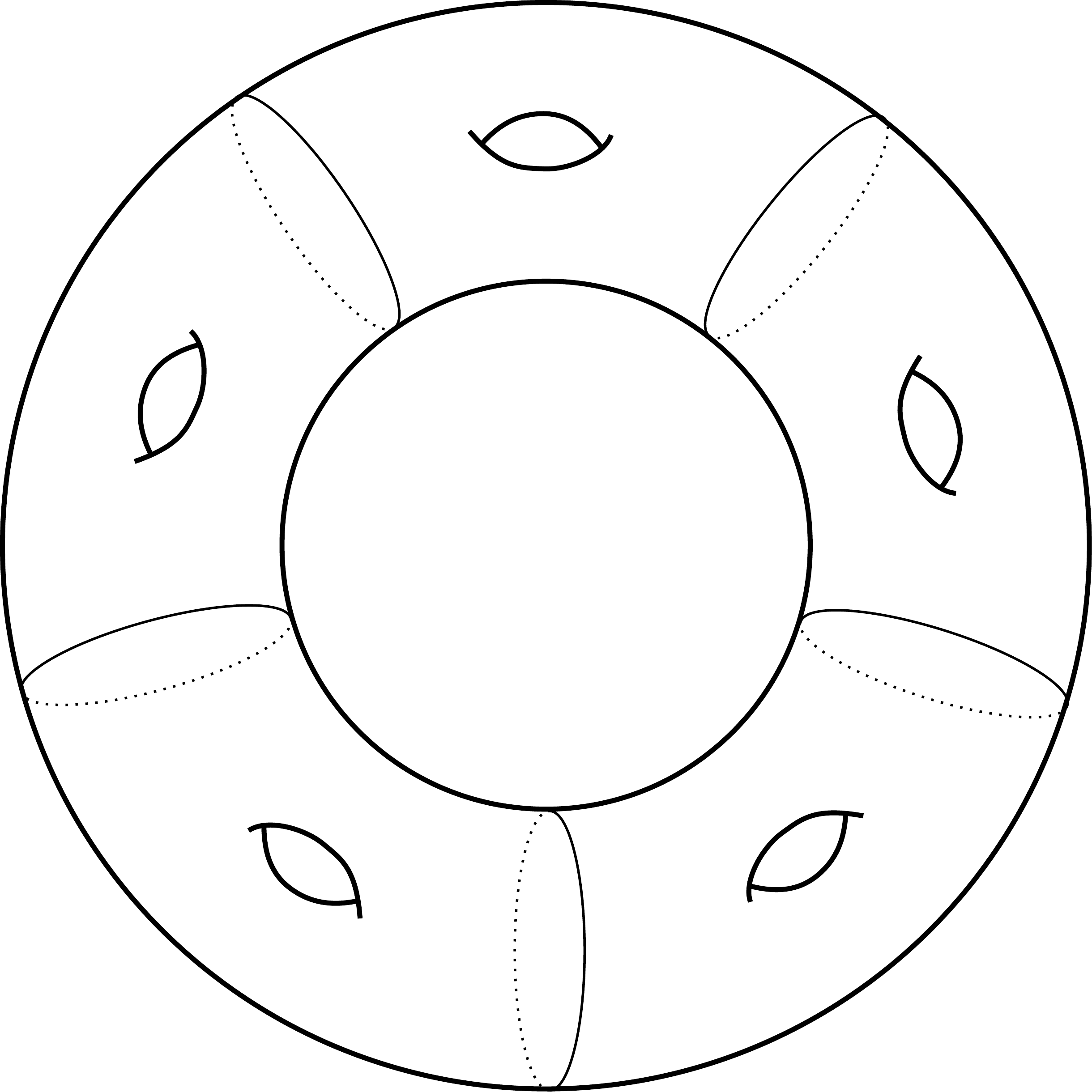}
	\subcaption{A maximal clique.}
	\label{maximal clique pic}
	\end{minipage}\hfill
	\caption{Subgraphs of $\cc_v(S_6)$.}
	\label{subgraphs pics} 
\end{figure}

The idea behind the proof of Theorem \ref{thm: cc_v(S_g)} is simple, but it takes some work to establish, which we bundle into a few Lemmas in covering space theory.
Suppose that $c$ and $d$ are homologous, oriented, simple closed curves in $S_g$.
Define an {\em immersion (of genus $h$) from $c$ to $d$} to be an immersion $i : S^2_h \to S_g$ that maps one oriented boundary component injectively to $-c$ and the other to $d$.

\begin{Lem}
\label{lem: immerse exist}
There exists an immersion between any pair of homologous, oriented, non-separating simple closed curves on $S_g$.
\end{Lem}

\begin{proof}
Suppose that $c$ and $d$ are simple closed curves that represent the class $v \in H_1(S_g;\bz)$, $v \ne 0$.
The epimorphism $\iota(v,-) : H_1(S_g;\bz) \to \bz$ defines a corresponding infinite cyclic cover $p_v : S_\infty \to S_g$.
This cover admits the following construction, which is reminiscent of, but simpler than, the construction of the infinite cyclic cover of a knot exterior.
Cut open $S_g$ along an open regular neighborhood of $c$ to form the surface $S \smallsetminus N(c) \approx S^2_{g-1}$,
whose oriented boundary consists of two components, labeled $c^+$ and $c^-$.
Form $(S \smallsetminus N(c)) \times \bz$, and identify $c^+ \times \{n\}$ with $c^- \times \{n+1\}$ by an orientation-reversing homeomorphism for all $n \in \bz$.
The resulting quotient space is $S_\infty$, and it comes equipped with a covering map $p_v$ from its construction.
Fix a lift $\widetilde{c}$ to $S_\infty$ and let $t$ denote a generator of the deck transformation group.
The lift $\widetilde{c}$ separates $S_\infty$ into two components $S_\infty^\pm$, labeled so that $t \cdot S_\infty^+ \subset S_\infty^+$.
Since $[d] = v$, we can choose a lift of $d$ to a simple closed curve $\widetilde{d} \subset S_\infty^+$.
The lifts $-\widetilde{c}$ and $\widetilde{d}$ cobound a compact subsurface $T \subset S_\infty^+$, $T\approx S^2_h$ for some $h \ge 0$.
The desired immersion $i$ is simply the restriction of the covering map $p_v$ to $T$.
\end{proof}

The following Lemma shows that every immersion arises in the manner of Lemma \ref{lem: immerse exist}.
It also plays a role in Lemma \ref{lem: immerse unique}.

\begin{Lem}
\label{lem: immerse lift}
An immersion between two curves representing the homology class $v \ne 0$ lifts to an embedding in the infinite cyclic cover $p_v : S_\infty \to S_g$.
\end{Lem}

\begin{proof}
Let $i : S^2_h \to S_g$ denote the immersion.
Write $\del S^2_h = -\del_1 \cup \del_2$ and $c_j = i(\del_j)$, $j=1,2$.
Choose lifts $\widetilde{c_j}$ of each to $S_\infty$ with the property that $\widetilde{c_2} - \widetilde{c_1}$ is the oriented boundary of a compact subsurface $T \subset S_\infty$.
Excise the interior of $T$ and splice in $S^2_h$ by the rule that a point $x \in \widetilde{c_j}$ glues to the unique point $y \in \del_j$ with the property that $p_v(x) = i(y)$.
The result is a surface $S$ with a map $p: S \to S_g$.
Moreover, $p$ is a covering map by the construction and the assumption that $i$ is an immersion.
If a loop $\gamma \subset S_g$ lifts to a loop in $S$, then it lifts to a loop supported in $S \smallsetminus S^2_h$, since $S^2_h$ is compact and $S$ is infinite-sheeted.  As $S \smallsetminus S^2_h$ is homeomorphic to $S_\infty \smallsetminus T$ by a map that respects the covering maps, it follows that $\gamma$ lifts to a loop in $S_\infty$.
Thus, $\mathrm{im}(p_*) \subset \mathrm{im}((p_v)_*)$, and by symmetry, we have $\mathrm{im}(p_*) = \mathrm{im}((p_v)_*)$.
It follows that there exists a covering space isomorphism $\varphi: S \to S_\infty$.
Composing the inclusion from $T$ to $S$ with $\varphi$ produces the required embedding.
\end{proof}

\begin{Lem}
\label{lem: immerse unique}
Any two immersions between a pair of non-separating simple closed curves on $S_g$ have the same genus$\pmod{g-1}$, and this value depends only on the isotopy classes of the curves.
\end{Lem}

\begin{proof}
Suppose that $c$ and $d$ are homologous simple closed curves on $S_g$ and we have two immersions from $c$ to $d$.
By Lemma \ref{lem: immerse lift}, both lift to embeddings under $p_v : S_\infty \to S$.
By applying a deck transformation, we may assume that both embeddings map one boundary component to a fixed lift $\widetilde{c}$.
Denote by $\widetilde{d}$ and $t^k \cdot \widetilde{d}$ the images of the other boundary components, $k \ge 0$.
The images of the two embeddings differ by the subsurface of genus $k \cdot (g-1)$ cobounded by 
$-\widetilde{d}$ and $t^k \cdot \widetilde{d}$.
It follows that the immersions have the same genus$\pmod{g-1}$, as desired.
The fact that this value depends only on the isotopy classes of $c$ and $d$ follows by an easy argument using the isotopy extension principle.
\end{proof}

Together, Lemmas \ref{lem: immerse exist} and \ref{lem: immerse unique} promote the second part of \cite[Lemma 2]{Irmerchillingworth} to immersions.

\begin{proof}[Proof of Theorem \ref{thm: cc_v(S_g)}]
We build a clique of size $g-1$ in $\cc_v(S_g)$ by cyclically gluing together $g-1$ copies of the surface $S^1_1$ along their boundaries and taking the images of the boundary curves.
See Figure \ref{maximal clique pic}.

We obtain an optimal coloring $f: \cc_v(S_g) \to \bz / (g-1) \bz$ as follows.
Fix an oriented simple closed curve $c \subset S$ representing the class $v$.
Given an oriented simple closed curve $d \subset S$ that also represents $v$, there exists an immersion from $c$ to $d$, by Lemma \ref{lem: immerse exist}.
Its genus $h \pmod{g-1}$ depends only on the isotopy type of $d$, by Lemma \ref{lem: immerse unique}.
We may therefore unambiguously define $f(d) \equiv h \pmod{g-1}$.

To show that the coloring $f$ is proper, suppose that $d_1$ and $d_2$ are disjoint simple closed curves that represent the class $v$.
Thus, there exists a subsurface $T \subset S$ with oriented boundary $d_2 - d_1$ and genus $0 < t < g-1$.
Stacking $T$ onto an immersion $i : S^2_h \to S$ from $c$ to $d_1$ gives an immersion from $c$ to $d_2$ of genus $h + t \not\equiv h \pmod{g-1}$.
It follows that $f(d_1) \ne f(d_2)$, so the coloring is proper, as desired.
\end{proof}


\subsection{Domains}
\label{subsec: domains}

The coloring $f$ described in the proof of Theorem \ref{thm: cc_v(S_g)} admits an alternate description in terms of {\em domains} on $S$ that we will make use of in \S \ref{sec: permuting colors} (see Lemmas \ref{cyclic permutation} and \ref{lemma: chi homo}).
Domains and their combinatorial Euler measures recur throughout Heegaard Floer homology; see \cite[\S2]{Sarkar2011} for a quick, thorough treatment.
In our setting, this description allows us to compute $f$ directly on $S$, without passing to the cover $S_\infty$.

As in the proof of Theorem \ref{thm: cc_v(S_g)}, fix an oriented simple closed curve $c$ that represents $v$, and choose another oriented simple closed curve $d$ that also represents $v$.
Position $c$ and $d$ to meet transversely.
The complement $S - c - d$ consists of a number of connected components whose closures are called {\em regions}.
The classes of the regions form a basis for $H_2(S,c \cup d;\bz)$.
An element of this group is called a {\em domain}.
The neighborhood of each intersection point in $c \cap d$ contains four corners.
For each region $R$, let $e(R)$ denote the Euler characteristic of its interior and $c(R)$ its number of corners.
We define the combinatorial Euler measure of a region $R$ by $m(R) = e(R) - c(R)/4$.
Its definition extends to domains by linearity.
Since $c$ and $d$ are homologous, there exists a domain $D$ with $\del D = d-c$.
The domain is well-defined up to multiples of $[S] = \sum_R [R]$.
It follows that the value $m(D)$ is well-defined modulo $m([S]) = e(S) = -2(g-1)$.
Define $f'(d) = m(D) \, (\mathrm{mod} \, 2(g-1))$.

\begin{Prop}
\label{prop: measure}
The maps $f$ and $f'$ obey the relation $2f=f'$.
\end{Prop}

\begin{proof}
Choose a curve $d$ representing the class $v$.
The decomposition of $S$ into regions induced by $c \cup d$ lifts to one of $S_\infty$.
Choose a subsurface $\Sigma \subset S_\infty$ with $\del \Sigma = \widetilde{d} - \widetilde{c}$.
The projection $p_v(\Sigma)$ is a domain in $(S,c \cup d)$ with boundary $d - c$.
We have $2f(d) \equiv 2g(\Sigma) = m(\Sigma) = m(p_v(\Sigma)) \equiv f'(d) \pmod{2(g-1)}$, using the additivity of $m$ in the third equality.
\end{proof}

As a corollary to Proposition \ref{prop: measure}, the definition of $f'$ descends to isotopy classes of curves.
We could instead establish this fact directly from the definition of $f'$ and an application of the bigon criterion.
We could also base a proof of Theorem \ref{thm: cc_v(S_g)} on the definition of $f'$ instead of $f$.
In particular, $f'$ is a proper coloring, although it is less immediate from its definition that it only takes even values, which is required to see that it is a proper $(g-1)$-coloring and not just a proper $2(g-1)$-coloring.

\begin{Rem}
\label{Irmer remark 1}
In a slightly different guise, the map $f'$ was considered by Irmer: it is twice the \emph{signed length} from the chosen basepoint $c\in \cc_v(S)$ \cite[Lemma 2]{Irmerchillingworth}.
We note that Irmer's description of the computation of $f'(d)$ requires a path from $c$ to $d$ in $\cc_v(S_g)$, whereas ours requires only a realization of $c$ and $d$ by transverse simple closed curves.
For instance, for the curves $c$ and $d$ pictured in Figure \ref{fig:domainColorExample}, $f'(d)\equiv 0$.
In particular, $c$ and $d$ are curves of the same color.
\end{Rem}

\begin{figure}[h]
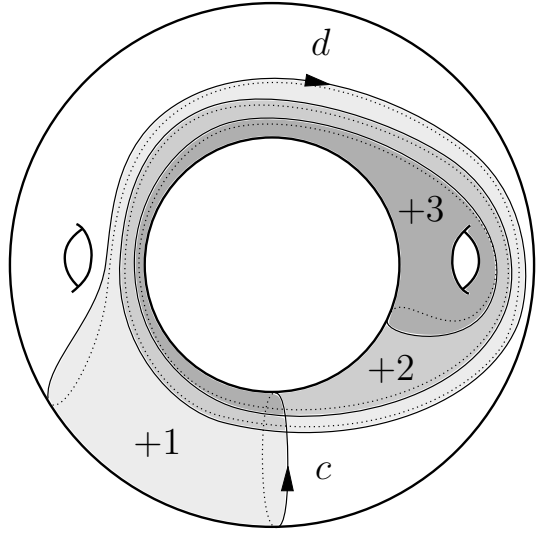

	\centering
	\Large
	\begin{lpic}{domainColorExample-new(7cm)}
		\lbl[]{148,27;$c$}
		\lbl[]{147,229;$d$}
		\lbl[]{70,40;$+1$}
		\lbl[]{180,75;$+2$}
		\lbl[]{194,150;$+3$}
	\end{lpic}
	\caption{The sum of the shaded regions, weighted according to labels, form a domain whose boundary is $d-c$. As its combinatorial Euler measure is $-6\equiv 0 \pmod2$, $c$ and $d$ have the same color.}
	\label{fig:domainColorExample}
\end{figure}

\begin{Rem} Given the difference in the order of growth of the chromatic numbers of $\cs(S)$ and $\cc_v(S)$, $v \ne 0$, the reader may wonder how the $(g-1)$-coloring of $\cc_v(S)$ fails for $v=0$. 
For example, even when $g=3$, we have $\omega(\cs(S_3))=3 > 2 = \chi(\cc_v(S_3))$.
The issue is explained in part by the fact that Lemma \ref{lem: immerse exist} fails in this setting: 
there does not exist a consistent way to orient the separating curves so that there exists an immersion between any two, as the reader may check on a maximum clique in $\cs(S_3)$.
Moreover, there exist domains of Euler measure 0 cobounded by disjoint oriented separating curves: take the image of an injective immersion $S_1^1 \sqcup S^1_1 \to S_g$, $g \ge 3$, that reverses orientation on one component.
\end{Rem}


\subsection{The upper bound on $\chi(\cc(S_g))$.}
\label{subsec: upper bound}

In this subsection, we use Theorem \ref{thm: cc_v(S_g)} to obtain the upper bound in Theorem \ref{thm: s_g intro}.
First, we extend Theorem \ref{thm: cc_v(S_g)} to the graph $\cc_{\overline{v}}(S_g)$ induced on curves that can be oriented to represent a fixed non-zero primitive homology class $\overline{v} \in H_1(S_g;\bz/m\bz)$, where $m > 1$ is a positive integer.

\begin{Prop}
\label{prop: (mod m)}
$\chi(\cc_{\overline{v}}(S_g)) = g-1$.
\end{Prop}

\begin{proof}
Let $V \subset H_1(S_g;\bz)$ denote the set of elements that reduce to $\overline{v} \pmod m$.
If $c$ and $d$ are disjoint representatives of $\overline{v}$, then they represent the same element $v \in V$.
It follows that $\cc_{\overline{v}}(S_g)$ is the disjoint union of the induced subgraphs $\cc_v(S_g)$, $v \in V$.
The $(g-1)$-coloring $\cc_v(S_g) \to \bz / (g-1) \bz$ described in the proof of Theorem \ref{thm: cc_v(S_g)} depends on a choice of representative $c_v \in \cc_v(S_g)$ for each $v \in V$.
A priori, the simultaneous existence of these representatives depends on the axiom of (countable) choice.
However, we can remove this dependence as follows.
Meeks and Patrusky establish an algorithm to produce a curve $c_v$ representing the class $v$ \cite[\S~1]{MeeksPatrusky} (see also \cite[Prop.~6.2]{Farbprimer}).
The algorithm depends on a fixed choice of geometric symplectic basis for $S_g$ and finitely many paths in $S_g$, and these choices are independent of $v$.  Taking the representative $c_v$ of $v$ output by the algorithm for all $v \in V$ produces the desired explicit set of representatives.
Their existence produces a simultaneous proper $(g-1)$-coloring of all $\cc_v(S_g)$, $v \in V$, whence the desired proper $(g-1)$-coloring of $\cc_{\overline{v}}(S_g)$.
\end{proof}

Recall the subgraph $\cn(S_g)$ of $\cc(S_g)$ induced on the nonseparating curves.

\begin{Cor}
\label{cor: chi non-sep}
$\chi(\cn(S_g)) \le (g-1) \cdot (2^{2g}-1)$.
\end{Cor}

\begin{proof}
Give each subcomplex $\cc_{\overline{v}}(S_g)$, $\overline{v} \in H_1(S_g;\bz/2\bz)^\times$, a proper $(g-1)$-coloring by Proposition \ref{prop: (mod m)}, using a different color pallette for each.  As there are $2^{2g}-1$ elements in $H_1(S_g;\bz/2\bz)^\times$, the stated bound follows.
\end{proof}

\begin{proof}
[Proof of Theorem \ref{thm: s_g intro}]
Immediate from Theorem \ref{thm:separating} and Corollary \ref{cor: chi non-sep}.
\end{proof}


\subsection{Unique colorability of $\cc_v(S)$}
\label{subsec: uniqueness}

A graph is {\em uniquely $k$-colorable} if it admits a proper $k$-coloring and any two proper $k$-colorings are related by a bijection between the color sets.
In other words, the graph admits a unique partition into $k$ independent sets.
We establish the following complement to Theorem \ref{thm: cc_v(S_g)}:

\begin{Thm}
\label{thm: C_v uniquely colorable}
$C_v(S_g)$ is uniquely $(g-1)$-colorable.
\end{Thm}

Let $S = S^b_g$, $g \ge 1$, $b \in \{0,2\}$.
If $b=0$, let $v \in H_1(S;\bz)$ denote a primitive, non-zero class, and if $b = 2$, let $v$ denote the peripheral class.
Recall the maximal clique graph $\ck(\cc)$ of $\cc$, whose vertices consist of the maximal simplices in $\cc$, and where two maximal simplices are adjacent if they meet in a codimension-1 face.

\begin{Prop}
\label{prop: connected}
$\ck(\cc_v(S))$ is connected.
\end{Prop}

\begin{proof}
First, suppose that $b = 2$.
We show that $\ck(\cc_v(S))$ is connected in this case using Putman's technique \cite[Lemma 2.1]{Putmanconnectivity}.
Since $v$ is the peripheral class, the mapping class group $\PMod(S)$ acts on $\ck(\cc_v(S))$, and by the classification of surfaces, it does so transitively.
Next, consider Figure \ref{fig: unique}.
The Dehn twists about the red curves shown there constitute the Humphries generating set for $\PMod(S^2_{g-1})$ \cite[Fig.~4.10]{Farbprimer}.
The blue curves shown there constitute a maximal simplex $K$ in $\cc_v(S)$.
Observe that each red curve meets at most one blue curve, and the algebraic intersection number between any red and blue curve is 0.
It follows that a Humphries generator either preserves $K$ or moves it to another maximal simplex that meets $K$ in a codimension-1 face.
Therefore, \cite[Lemma 2.1]{Putmanconnectivity} applies to show that $\ck(\cc_v(S))$ is connected.

\begin{figure}
\includegraphics[width=5in]{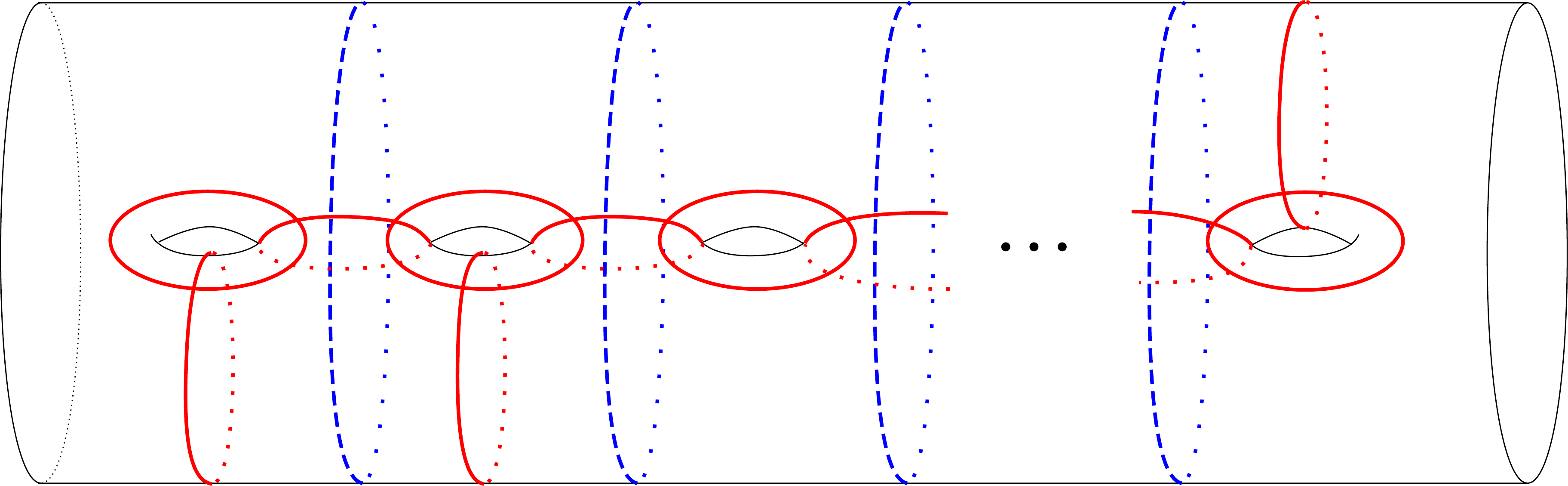}
\caption{The surface $S^2_g$.  The Dehn twists about the solid red curves are Humphries generators for $\PMod(S^2_g)$.  The dashed blue curves comprise a maximal simplex in $\cc_v(S^2_g)$, where $v$ denotes the peripheral homology class.}
\label{fig: unique}
\end{figure}

Next, suppose that $b=0$.
If $g =1$, then $\cc_v(S)$ is empty, and if $g = 2$, then $\cc_v(S)$ has no edges, and the desired result is trivial in either case.
Suppose then that $g \ge 3$.
Choose a pair of maximal simplices $K$ and $K'$ in $\cc_v(S)$, a curve $c$ in $K$, and a curve $c'$ in $K'$.
As $\cc_v(S)$ is connected \cite[Thm.~1.9]{Putmanconnectivity}, there exists a path $c=a_1, a_2, \dots, a_n=c'$ in $\cc_v(S)$ connecting $c$ and $c'$.
Thus, there exists a maximal simplex $K_i$ containing $a_i$ and $a_{i+1}$ for all $i=1,\dots,n-1$.
Additionally, set $K_0 = K$ and $K_n = K'$.
For all $i$, the link $\lk(a_i)$ is isomorphic to $\cc_{v'}(S^2_{g-1})$, where $v'$ denotes the peripheral class.
Therefore, there exists a sequence of maximal cliques in $\lk(a_i)$, beginning with $K_{i-1} \cap \lk(a_i)$ and ending with $K_i \cap \lk(a_i)$, such that any two in sequence meet in a codimension-1 face.
Taking the joins of these simplices with $\{a_i\}$ results in a path from $K_{i-1}$ to $K_i$ in $\ck(\cc_v(S))$.
The concatenation of these paths is a path from $K$ to $K'$ in $\ck(\cc_v(S))$.
Therefore, $\ck(\cc_v(S))$ is connected.
\end{proof}

A simplicial complex is \emph{pure of dimension $l$} if every simplex is contained in an $l$-simplex.
The proof of the following result is straightforward and left as an exercise to the reader.

\begin{Lem}
\label{lem: unique}
If a simplicial complex $\cc$ is pure of dimension $k-1$ and properly $k$-colorable, and $\ck(\cc)$ is connected, then $\cc$ is uniquely $k$-colorable. \qed
\end{Lem}

\begin{proof}[Proof of Theorem \ref{thm: C_v uniquely colorable}]
Theorem \ref{thm: cc_v(S_g)} shows that $\chi(C_v(S_g)) = g-1$.
Proposition \ref{prop: connected} shows that $\ck(\cc_v(S))$ is connected.
Lemma \ref{lem: unique} now gives the desired result.
\end{proof}


\subsection{Surfaces with boundary}
\label{subsec: with boundary}

We now turn to the case of a surface with boundary $S = S_g^b$, $b \ge 1$.
More cases arise due to the distinction of whether a separating homology class is null-homologous, peripheral, or not.
As with a closed surface, we find a qualitative difference between $v=0$ and $v \ne 0$.
However, for each type of $v$, the growth of $\chi(\cc_v(S))$ is governed by $g$ and not by $b$.

\begin{Thm}
\label{thm: nullhomologous}
$\chi(\cc_0(S_g^b)) = \Theta(g \cdot \log g)$.
\end{Thm}

\begin{proof}
Given $c\in \cc_0(S)$, let $F(c)$ denote the component of $S\smallsetminus c$ containing $\del S$.  Let $G_0(S)$ denote the subgraph induced on the curves $c$ for which $F(c)$ has  genus 0, and let $G_+(S)$ denote the subgraph on curves for which $F(c)$ has positive genus.
Note that $G_0(S)$ is an independent set in $\cc_0(S)$, and it consists of precisely the curves in $\cc_0(S)$ whose images are inessential under the inclusion $i: S \into S_g$ obtained by capping off each component of $\del S$ by a disk.
It follows that $i$ induces a map $f: G_+(S) \to \cs(S_g)$.
If $c$ and $d$ span an edge in $G_+(S)$, then $S \smallsetminus c \cup d$ consists of three components of positive genus, and the same is true of $S_g \smallsetminus f(c) \cup f(d)$.
It follows that $f$ is a homomorphism.
(However, $f$ is typically {\em not} injective: distinct curves in $G_+(S)$ related by a Dehn twist about a curve in $G_0(S)$ will have the same image under $f$.)
Since the chromatic number is monotone under homomorphisms, it follows that $\chi(G_+(S)) \le \chi(\cs(S_g)) = O(g \cdot \log g)$, and using one more color on $G_0(S)$ gives the required upper bound on $\chi(\cc_0(S))$.

For the lower bound, embed a planar surface $\Sigma_{g+1} \into S$ so that one component of $\del \Sigma_{g+1}$ is in $G_0(S)$ and all of the others bound disjoint subsurfaces $S_1^1 \subset S$.
The inclusion induces an embedding $\cc(\Sigma_{g+1}) \into \cc_0(S)$, and Theorem \ref{thm:planar} supplies the lower bound.
\end{proof}

\begin{Thm}
\label{thm: non-nullhomologous}
For $v \ne 0$, $\cc_v(S_g^b)$ is uniquely $t$-colorable, where
\[
t=
\begin{cases}
g-1, & \textup{if } v \textup { is separating and } b=2; \\
g, & \textup{if }  v \textup{ is non-separating, or if } v \textup{ is peripheral and } b > 2; \\
g+1, & \textup{if } v \textup{ is non-peripheral and separating.}
\end{cases}
\]
Moreover, $t = \chi(\cc_v(S_g^b)) = \omega(\cc_v(S_g^b))$.
\end{Thm}

\begin{proof}[Proof sketch.]
First, we describe the optimal coloring.
For the case of a separating class $v$, fix a boundary component $\del_0 \subset \del S$, and color a curve $c \in C_v(S)$ by the genus of the subsurface of $S \smallsetminus c$ containing $\del_0$.
For the case of a non-separating class $v$, we adapt the coloring described in Theorem \ref{thm: cc_v(S_g)}.
Now the lifts $\widetilde{c}$ and $\widetilde{d}$ cobound a compact subsurface $\widetilde{T} \subset S_\infty$ with some components of $\del S_\infty$.
We color $d$ by adding the genus of $\widetilde{T}$ to the number of preimages of $\del_0$ in $\del \widetilde{T}$ and reducing$\pmod g$.
In each case, it is straightforward to check that the coloring so described is proper, uses the stated number of colors, and that there exists a clique of the stated size.  Figure \ref{fig:homologous-separating} displays a maximal clique in the last case.
The proof of unique colorability proceeds along the lines of in \S \ref{subsec: uniqueness} with minor changes, using the generating set for $\PMod(S^b_g)$ displayed in \cite[Fig.~4.10]{Farbprimer}.
\end{proof}

\begin{figure}[h]
\centering
\includegraphics[width=11cm]{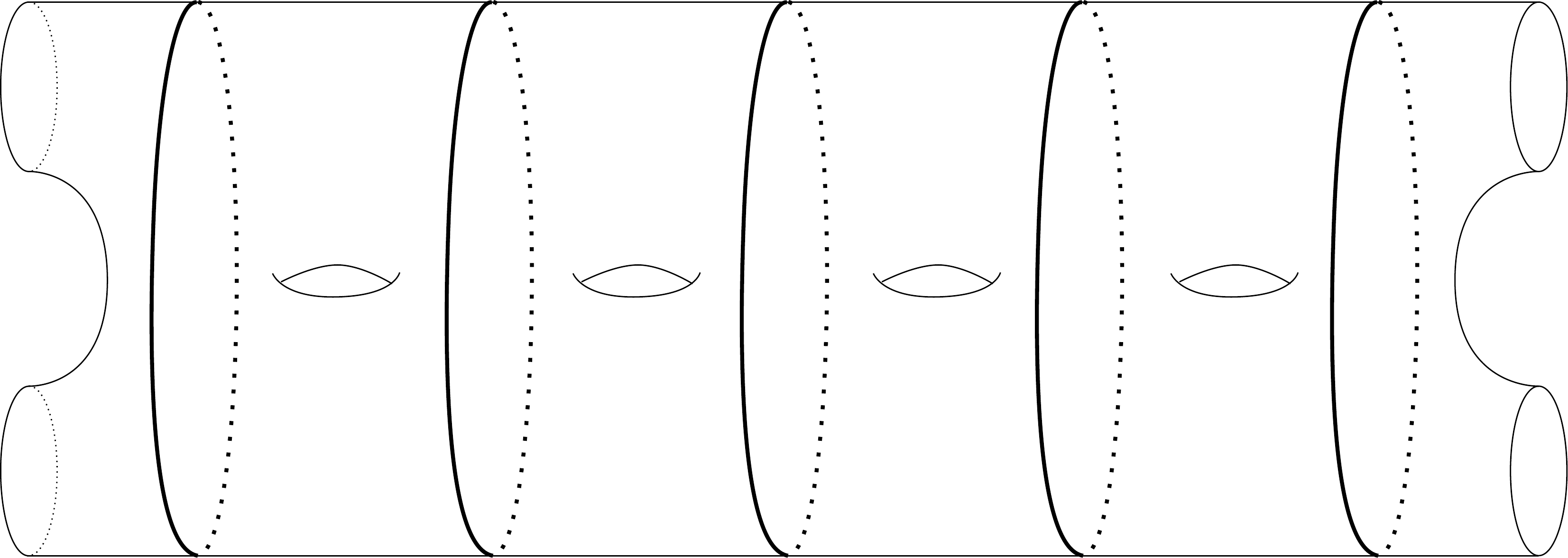}
\caption{
A $(g+1)$-clique in $\cc_v(S_g^b)$, $v$ a non-peripheral separating class.
}
\label{fig:homologous-separating}
\end{figure}


\section{Permuting the colors and the Johnson homomorphism}
\label{sec: permuting colors}

Throughout this section, we restrict to the case of a closed surface $S=S_g$.
Because $\cc_v(S)$ is uniquely $(g-1)$-colorable, any automorphism of it must permute the colors in a $(g-1)$-coloring. 
We therefore obtain a homomorphism from the automorphism group $\Aut(\cc_v(S))$ to the symmetry group of the colors. 
In this section, we investigate the restriction of this homomorphism to the Torelli group, which arises as a subgroup of $\Aut(\cc_v(S))$. 
As a result, we find a precise relationship between the permutation of the colors and the well-studied Johnson homomorphism. 
For background and references on the Torelli group and the Johnson homomorphism, see \cite[\S\S6.5-6.6]{Farbprimer} and \cite{Putmanjohnson}.


\subsection{The action of the Torelli group on the color classes}
\label{subsec: action of Torelli}
Set $H = H_1(S;\bz)$, let $v \in H$ denote a primitive, non-zero class, and let $\ci < \Mod(S)$ denote the {\em Torelli group}, the mapping classes that act trivially on $H$.
Because $\ci$ acts by graph automorphisms of $\cc_v(S)$, Theorem \ref{thm: C_v uniquely colorable} implies that we obtain a homomorphism $\chi_v:\ci \to \Sym(g-1)$ that records the permutation of the colors in a $(g-1)$-coloring of $\cc_v(S)$.
Fix an orientation-preserving homeomorphism $\phi$ representing a mapping class in $\ci$.
Note that the coloring $f: \cc_v(S) \to \bz / (g-1)\bz$ produced in Theorem \ref{thm: cc_v(S_g)} required the choice of a fixed oriented curve $c\in \cc_v(S)$.
The following result shows that the permutation induced by $\phi$ does not depend on this choice, and that it permutes the colors cyclically:

\begin{Lem}
\label{cyclic permutation}
Let $d$ be any oriented simple closed curve representing $v \in H$, and let 
$C(d, \phi \cdot d)$ be a domain satisfying  $\partial C(d, \phi \cdot d) = \phi \cdot d - d$. 
The permutation $\chi_v(\phi)$ shifts every color by $\frac 12 m( C(d, \phi \cdot d)) \pmod{g-1}.$
\end{Lem}

Recall that $m(C)$ denotes the combinatorial Euler measure of the domain $C$. 

\begin{proof}
Fix an oriented curve $c \in \cc_v(S)$. 
By Proposition \ref{prop: measure}, the coloring $f:\cc_v(S)\to \bz / (g-1)\bz$ is given by 
\[
f(\gamma) = \frac{1}{2} m(C(c,\gamma))\pmod{g-1},
\]
for any curve $\gamma$ representing $v$ and domain $C(c,\gamma)$ satisfying $\del C(c,\gamma) = \gamma - c$.
Let $C(\gamma, \phi\cdot \gamma)$ be a domain with boundary $\phi\cdot \gamma-\gamma$, and note that $C(c,\gamma) + C(\gamma, \phi\cdot \gamma)$ is a domain with boundary $\phi \cdot \gamma - c$. 
Since $m$ is additive, it follows that 
\[
f( \phi \cdot \gamma) = f(\gamma) + \frac{1}{2}m(C(\gamma, \phi\cdot \gamma))  \pmod{g-1}~.
\]

Next, choose a domain $C(d,\gamma)$ with boundary $\gamma-d$. 
Because $\phi$ preserves Euler measure, the domain $C(d,\gamma) + C(\gamma,\phi\cdot \gamma) - \phi\cdot C(d,\gamma)$ has Euler measure $m(C(\gamma,\phi\cdot \gamma))$.
Its boundary is evidently $\phi\cdot d -d$, and it follows by Proposition \ref{prop: measure} that $m(C(\gamma, \phi\cdot \gamma)) \equiv m(C(d,\phi\cdot d)) \pmod{g-1}$.
Thus, $\phi$ acts as claimed on color classes.
\end{proof}

By Lemma \ref{cyclic permutation}, we may record the homomorphism induced by the color permutation by the value $\chi_v (\phi) = \frac{1}{2} m(C(d,\phi\cdot d)) \in \bz / (g-1)\bz$, where $[d]=v$. 
We thereby obtain a map $\chi(\phi): v \mapsto \chi_v(\phi)$ defined on the primitive non-zero classes $v\in H$.

\begin{Lem}
\label{lemma: chi homo}
The map $\chi(\phi)$ extends to a homomorphism $\chi(\phi) : H \to \bz / (g-1)\bz$.
\end{Lem}

\begin{proof}
The material of \S \ref{subsec: domains} concerning domains and their Euler measures readily generalizes from oriented simple closed curves to multicurves.
In particular, for a multicurve $\gamma \in C_1(S;\bz)$ and an element $\phi \in \ci$, there exists a domain $C$ such that $\del C = \gamma - \phi \cdot \gamma$.
It is well-defined up to adding a multiple of $[S]$, so we obtain a well-defined value $\chi(\phi)(\gamma) = \frac 12 m(C)\pmod{g-1}$ as before.
It is clear that $\chi(\phi)$ is linear on multicurves, so it defines a homomorphism $C_1(S;\bz) \to \bz/(g-1)\bz$.
If $\gamma$ is null-homologous, then there exists a domain $D$ such that $\del D = \gamma$, so $D - \phi \cdot D$ is a domain with boundary $\gamma - \phi \cdot \gamma$, and we obtain $\chi(\phi)(\gamma) = m(D - \phi \cdot D) = m(D) - m(\phi \cdot D) = 0$.
It follows that $\chi(\phi)$ descends to a homomorphism $H \to \bz / (g-1) \bz$ that extends its definition on primitive classes, as desired.
\end{proof}

Note that $\chi(\phi)$ depends only on the mapping class of $\phi$ and that $\chi(\phi \circ \psi) = \chi(\phi) + \chi(\psi)$.
Thus, we obtain a homomorphism $\chi : \ci \to \mathrm{Hom}(H,\bz / (g-1) \bz)$ that records the simultaneous permutation of the $(g-1)$-colorings of all of the graphs $\cc_v(S)$.
Since $\chi$ has abelian image, the knowledgable reader may rightly suspect that it factors through the Johnson homomorphism $\tau$ for $g \ge 3$.
Indeed, the kernel of $\tau$ is generated by Dehn twists about separating curves for $g \ge 3$ \cite{Johnsonstructure2}, so it suffices to check that $\chi$ vanishes on these.
Given a Dehn twist $\phi$ about a separating curve $c$, choose a geometric basis for $S$ disjoint from $c$.
Lemmas \ref{cyclic permutation} and \ref{lemma: chi homo} demonstrate that $\chi(\phi)$ is trivial by checking its action on this basis, certifying that $\chi$ factors through $\tau$.
In the next subsection, we identify $\chi$ with a precursor to $\tau$ (Theorem \ref{thm: coloring permutation}).
We could conclude this result at once on the basis of Remark \ref{Irmer remark 1} and Irmer's work on the relationship of signed stable length with $\tau$ \cite{Irmerchillingworth}, but we include more detail for completeness.


\subsection{The coloring permutation and the Chillingworth homomorphism}
\label{Johnson homo}

Chillingworth introduced the \emph{winding number} of a regular closed curve around a non-vanishing vector field and applied its study to $\Mod(S)$ \cite{ChillingworthwindingI,ChillingworthwindingII}.
Johnson broadened and clarified Chillingworth's work, reworking the former's construction into the \emph{Chillingworth homomorphism} $t$ and showing that it specializes the \emph{Johnson homomorphism} $\tau$ \cite{Johnsonabelian}.
In the case of a closed surface $S$, these maps take the forms
\[
t:\ci \to\Hom(H,\bz/(g-1)\bz) \quad \mathrm{and} \quad \tau:\ci \to\Hom(H, \wedge^2 H) /j(H),
\]
where $j$ denotes a particular inclusion map.
An algebraic juggle involving $j$ and the algebraic intersection pairing $\iota$ on $H$ produces a map $\overline{\iota} : \Hom(H, \wedge^2 H) / j(H) \to \Hom(H,\bz / (g-1)\bz)$.
The Chillingworth and Johnson homomorphisms for a closed surface are then related by the composition $t = \overline{\iota} \circ \tau$ \cite[Thm.~2]{Johnsonabelian}.
Johnson also computed the image under $t$ of the \emph{bounding pair} maps, which generate $\ci$ for $g \ge 3$ \cite[Cor.~1\&\S6]{Johnsonabelian}:

\begin{Lem}
\label{lem: tau computation}
Suppose that $\alpha$ and $\beta$ form a bounding pair and $S\smallsetminus (\alpha\cup \beta)$ consists of components $\Sigma_1$ and $\Sigma_2$. 
The image under $t$ of the bounding pair map $\phi_{\alpha,\beta}$ is given by
$t(\phi_{\alpha,\beta})(v) = \mathrm{genus}(\Sigma_1) \cdot \iota(v,[\alpha])\pmod{g-1}$.\qed
\end{Lem} 

\begin{figure}
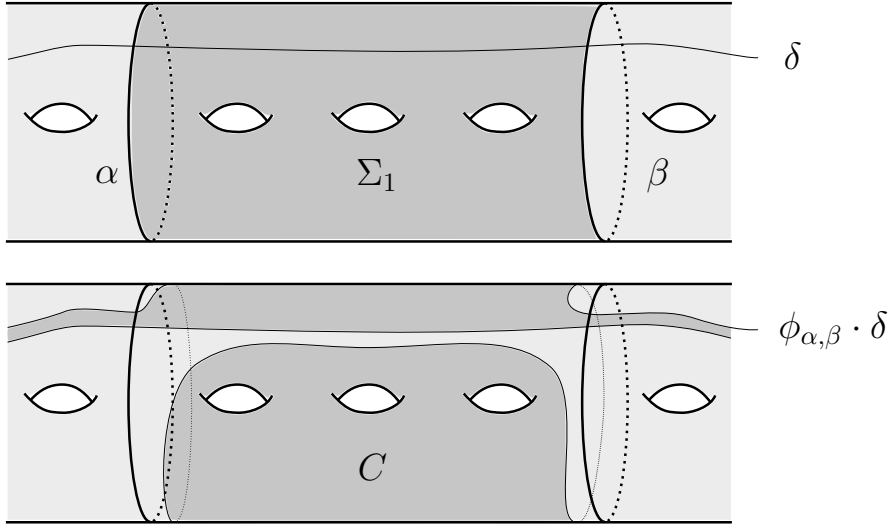

	\centering
	\Large
	\begin{lpic}{colorPerm-new(10cm)}
		\lbl[]{45,155;$\alpha$}
		\lbl[]{290,155;$\beta$}
		\lbl[]{165,155;$\Sigma_1$}
		\lbl[]{350,208;$\delta$}
		\lbl[]{368,86;$\phi_{\alpha,\beta}\cdot \delta$}
		\lbl[]{163,25;$C$}
	\end{lpic}
	\caption{
	The bounding pair $\alpha,\beta$ cuts off the subsurface $\Sigma_1$, the domain $C$ has boundary $\phi_{\alpha,\beta}\cdot \delta - \delta$, and the curve $\delta$ satisfies $\chi_{[\delta]}(\phi_{\alpha,\beta}) \equiv \mathrm{genus}(\Sigma_1)\pmod{g-1}$.}
	\label{fig: color perm}
\end{figure}

\begin{proof}[Proof of Theorem \ref{thm: coloring permutation}]
We follow Irmer's line of argument \cite[Theorem 1]{Irmerchillingworth}.
The assertion holds vacuously when $g \le 2$.
Since $\ci$ is generated by bounding pair maps for $g\ge 3$ \cite{Johnsonstructure1}, it suffices to check the assertion in this case on a bounding pair map $\phi_{\alpha,\beta}$.
Choose a geometric basis for $S$ consisting of $\alpha$, a curve $\delta$ meeting $\alpha$ and $\beta$ both once apiece, and all other curves disjoint from both $\alpha$ and $\beta$.
The fact that $\chi(\phi_{\alpha,\beta}) = t(\phi_{\alpha,\beta})$ is immediate from Lemmas \ref{cyclic permutation}, \ref{lemma: chi homo}, and \ref{lem: tau computation}. Figure \ref{fig: color perm} displays the single non-trivial case.
\end{proof}

As a corollary, we can recast the coloring of $\cc_v(S)$ in terms of the Torelli group and the Chillingworth homomorphism.
Fix $c\in \cc_v(S)$ and choose any $d \in \cc_v(S)$. 
Because $\ci$ acts transitively on $\cc_v(S)$, there exists some $\phi_d\in \ci_g$ so that $\phi_d \cdot c = d$.

\begin{Cor}
\label{cor: re-coloring}
The map $d \mapsto t(\phi_d)$ coincides with the $(g-1)$-coloring $f: \cc_v(S_g) \to \bz / (g-1) \bz$. \qed
\end{Cor}


\section{Arcs on planar surfaces}
\label{sec: arc graphs}

In this section, we obtain analogues of our results on curve graphs for arc graphs.
For simplicity, we specialize to the case of a planar surface $\Sigma = \Sigma_n$.

Let $v \in H_1(\Sigma,\del \Sigma;\bz/2\bz)$ denote a non-zero relative homology class, and let $\ca_v(\Sigma)$ denote the subgraph of $\ca(\Sigma)$ induced on arcs that represent $v$.
Note that $v$ is determined by the pair of (distinct) components of $\del \Sigma$ containing the endpoints of any arc that represents it.
Let $\ca\cn(\Sigma)$ consist of the non-separating arcs.
Every arc in $\ca\cn(\Sigma)$ belongs to one of the $\binom n2$ isomorphic subgraphs $\ca_v(\Sigma)$.
Similarly, for a hole $\del$ in $\Sigma$, let $\ca_\del(\Sigma)$ denote the subgraph of $\ca(\Sigma)$ induced on arcs with both endpoints on $\del$.
Let $\ca\cs(\Sigma)$ consist of the separating arcs.
Every arc in $\ca\cs(\Sigma)$ belongs to one of the $n$ isomorphic subgraphs $\ca_\del(\Sigma)$.

By analogy to Lemma \ref{lem:planar-upper} and with a slight adjustment to its proof, we have:

\begin{Lem}
\label{lem:arcs-upper}
There exists a homomorphism $f: \ca\cs(\Sigma_n) \to \kg(n)$.
\end{Lem}

\begin{proof}
Given an arc $a \in \ca\cs(\Sigma_n)$, let $\del$ denote the hole that contains its endpoints.
The surface $\Sigma_n \smallsetminus a$ consists of two components, each of which contains at least one hole of $\Sigma_n$.  
We obtain an induced partition of $\del \Sigma_n \smallsetminus \del$ into two non-trivial parts, and we let $f(a)$ be the partition of $\del \Sigma_n$ obtained by adding $\del$ to the larger part, breaking ties arbitrarily.
If $(a,b) \in \ca\cs(\Sigma_n)$, then there are two cases to check that $f(a,b) \in \kg(n)$, depending on whether $a$ and $b$ have endpoints on the same hole or not.
We leave the routine verification to the reader.
It follows that $f$ defines the desired homomorphism.
\end{proof}

Similarly, by analogy to Lemma \ref{lem:planar-lower}, we have:

\begin{Lem}
\label{lem:arcs-lower}
There exist embeddings $a_0 : \cg(n-1) \hookrightarrow \ca_\del (\Sigma_n)$ and $a_1 : \cg(n) \smallsetminus \cg(n,1) \hookrightarrow \ca \cn (\Sigma_n)$.
\end{Lem}

\begin{proof}
For the first part, realize $\Sigma_n$ as a round disk $D$ with $n-1$ evenly spaced small holes $\del_1,\dots,\del_{n-1}$ near $\del D$.
For $1 \le i \le j \le n-1$, let $a_0(i,j)$ denote a chord of $D$ disjoint from these holes and that partitions them into two parts, one of which consists of $\del_i,\dots,\del_{j-1}$.
Note that the isotopy type of $a_0(i,j)$ is well-defined.
Furthermore, $a_0(i,j)$ and $a_0(i',j')$ can be disjoined if and only if $(i,j)$ and $(i',j')$ are unlinked.
It follows that $a_0$ defines the desired embedding.

For the second part, realize $\Sigma_n$ as a sphere $\br^2 \cup \{ \infty \}$ with $n$ evenly spaced small holes $\del_1,\dots,\del_n$ around a circle.
For $1 \le i < j \le n$, let $a_1(i,j)$ denote a line segment with endpoints on $\del_i$ and $\del_j$.
As in the first part, it follows that $a_1$ defines the desired embedding.
\end{proof}

\begin{Thm}
\label{thm: separating arcs}
$\chi(\ca\cs(\Sigma_n)) = \Theta(n \log n)$.
\end{Thm}

\begin{proof}
The result follows from the sequence of homomorphisms 
\[
\cg(n-1) \overset {a_0} \into A_\del(\Sigma_n) \into \ca\cs(\Sigma_n) \overset f \to \kg(n)
\]
provided by Lemmas \ref{lem:arcs-upper} and \ref{lem:arcs-lower}, as well as one more application of Theorem \ref{thm: kneser chromatic intro} and the monotonicity of chromatic numbers.
\end{proof}

Doubling the surface $\Sigma_n$ along its boundary results in a closed surface $S_{n-1}$.
Every arc $a \in \ca(\Sigma_n)$ gets doubled to a curve $d(a) \in \cc(S_{n-1})$, and every class $v$ gets doubled to a non-zero class $w \in H_1(S_{n-1};\bz)$, defined up to sign.
The map $d: \ca(\Sigma_n) \to \cc(S_{n-1})$ defines an embedding, and it restricts to an embedding $\ca_v(\Sigma_n) \into \cc_w(S_{n-1})$.

\begin{Thm}
\label{thm: arcs unique}
$\chi(\ca_v(\Sigma_n)) = \omega(\ca_v(\Sigma_n)) = n-2$, and $\ca_v(\Sigma_n)$ is uniquely $(n-2)$-colorable.
\end{Thm}

\begin{proof}

\begin{figure}
	\centering
	\begin{minipage}{.45\textwidth}
	\centering
	\includegraphics[width=6cm]{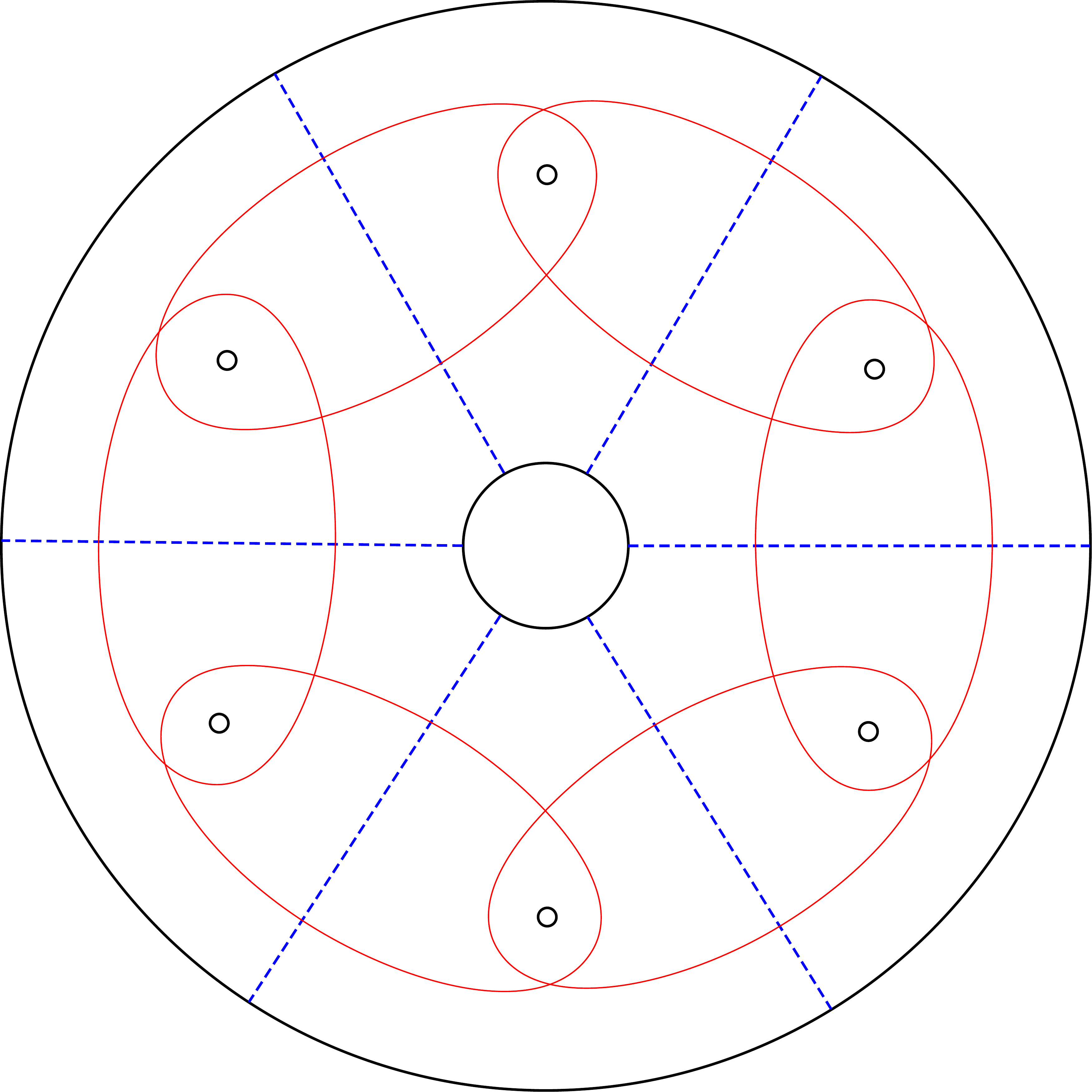}
	\subcaption{The dashed blue arcs form a maximal clique $c_0\in\ck(\ca_v(\Sigma_n))$.}
	\label{unique arcs 1 pic}
	\end{minipage}\hfill
	\begin{minipage}{.45\textwidth}
	\centering
	\includegraphics[width=6cm]{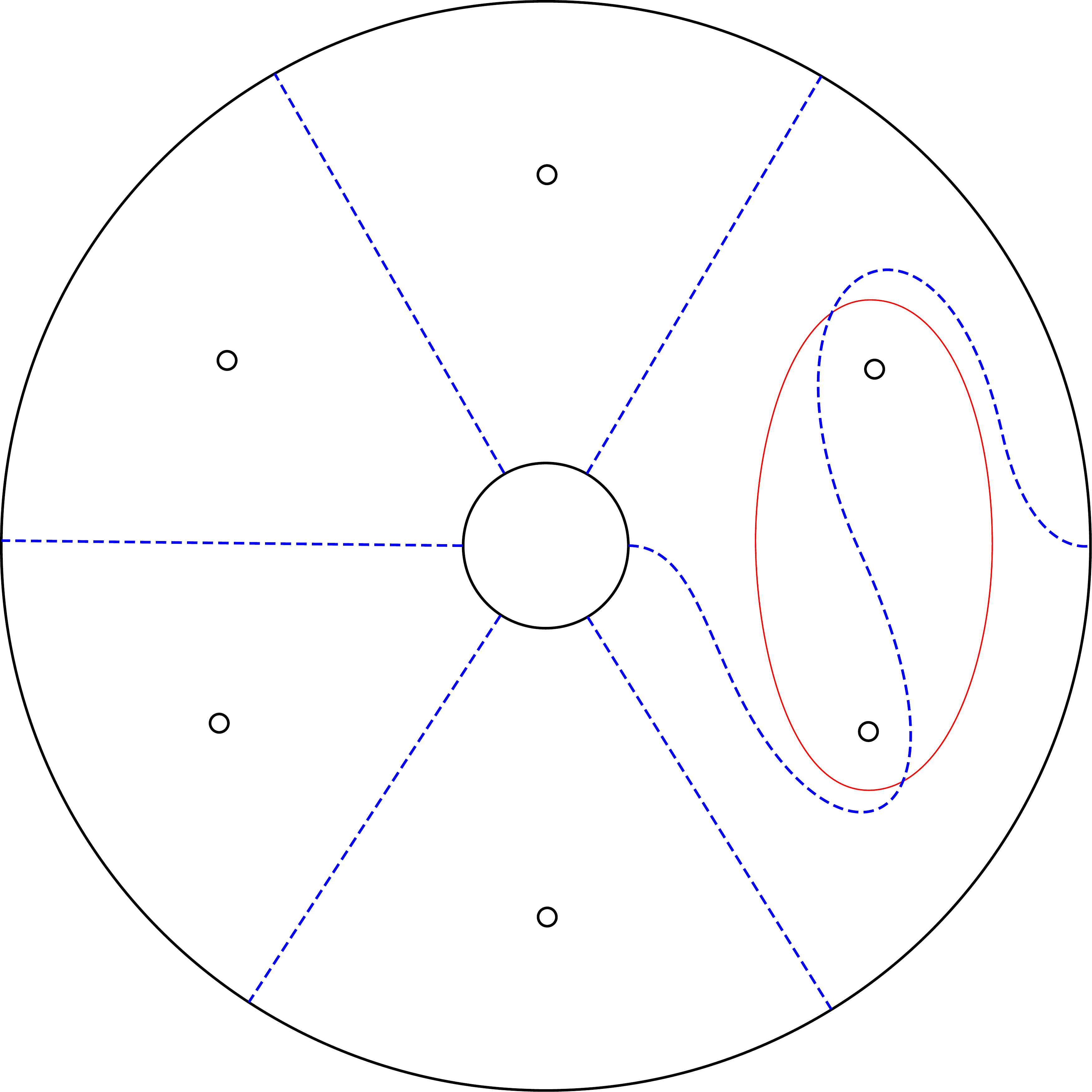}
	\subcaption{The action of a half Dehn twist on $c_0$.}
	\label{unique arcs 2 pic}
	\end{minipage}
	\caption{Half Dehn twists about the red curves on the left together with a rotation by one click generate the annular braid group.}
	\label{Unique arc pics} 
\end{figure}

The embedding $\ca_v(\Sigma_n) \into \cc_{w}(S_{n-1})$ and Theorem \ref{thm: C_v uniquely colorable} show that $\ca_v(\Sigma_n)$ is $(n-2)$-colorable.
The proof of uniqueness follows the paradigm of \S \ref{subsec: uniqueness} once more.
We realize $\Sigma_n$ as an annulus with $n-2$ equally spaced small holes.
Let $\del_i$ and $\del_j$ denote the boundary components of the annulus and let $v$ denote the class of arc with one endpoint on each of $\del_i$ and $\del_j$.
The subgroup $P_{ij}<\Mod(\Sigma_n)$ that fixes each of $\del_i$ and $\del_j$ setwise is the {\em annular braid group} on $n-2$ strands. 
It is generated by the $n-2$ swaps of consecutive holes by half Dehn twists and the rotation by one click \cite[Thm.~1]{Kentannularbraids}.
Choosing the clique $c_0\in\ck(\ca_v(\Sigma_n))$ given by $n-2$ radial segments, we see that the image of $c_0$ under each generator lies in the same connected component of $\ck(\ca_v(\Sigma_n))$ as $c_0$.
See Figure \ref{Unique arc pics}.
As $P_{ij}$ acts transitively on vertices of $\ck(\ca_v(\Sigma_n))$, Putman's technique shows that $\ck(\ca_v(\Sigma_n))$ is connected.
Lastly, the complex $\ca_v(\Sigma_n)$ is pure of dimension $n-3$, which also implies that $\omega(\Sigma_n)=n-2$.
Lemma \ref{lem: unique} closes the proof.
\end{proof}

\begin{proof}
[Proof of Theorem \ref{thm: planar arc graph intro}]
The arc graph $\ca(\Sigma_n)$ is the vertex-disjoint union of $\ca\cs(\Sigma_n)$ and the $\binom n2$ subgraphs $\ca_v(\Sigma_n)$.
The result now follows from the union bound and Theorems \ref{thm: separating arcs} and \ref{thm: arcs unique}.
\end{proof}

\begin{Rem}
The order of growth of the upper bound in Theorem \ref{thm: planar arc graph intro} is dominated by the nonseparating arcs, so it is tempting to explore more judicious colorings of nonseparating arcs of different types. 
Proposition \ref{prop:four-holed nosep} and Theorem \ref{thm: chi(a)=4} in the next \S take first steps in this direction.
\end{Rem}


\section{The four-holed sphere}
\label{sec: four-holed}

In this section, we study the arc graph of the four-holed sphere $\Sigma_4$ and obtain some exact results.
Although we cannot pin down the exact value of $\chi(\ca(\Sigma_4))$ (Theorem \ref{thm:four-holed arcs}), we determine the chromatic number of the subgraphs of $\ca(\Sigma_4)$ induced on separating arcs (Proposition \ref{prop:four-holed sep}), non-separating arcs (Proposition \ref{prop:four-holed nosep}), and non-separating arcs with one endpoint on a fixed hole (Theorem \ref{thm: chi(a)=4}).

\begin{Lem}
\label{lem: frac num four-holed}
$\chi_f(\ca(\Sigma_4)) \ge 22/3$.
\end{Lem}

\begin{proof}
Identify $\Sigma_4$ with the 2-skeleton of a regular tetrahedron in $\br^3$ with holes placed at its vertices. 
Let $V_1$ denote the arcs of $\Sigma_4$ determined by the edges of the tetrahedron. 
There are six reflections of $\Sigma_4$, and each has a fixed point set consisting of two arcs of $\Sigma_4$. 
From this pair, one arc is in $V_1$; let $V_2$ consist of the six other arcs of $\Sigma_4$ determined by these six reflections. Let $V_3$ consist of the twelve (isotopy classes of) arcs of $\Sigma_4$ determined by the intersection of $\Sigma_4$ with a plane passing through exactly one vertex of the tetrahedron. 
Finally, consider the subgraph $G$ induced on the vertex set $V=V_1 \cup V_2 \cup V_3$. See Figure \ref{tetra pics}.

\begin{figure}[h]
	\centering
	\begin{minipage}{.3\textwidth}
	\centering
	\includegraphics[width=4cm]{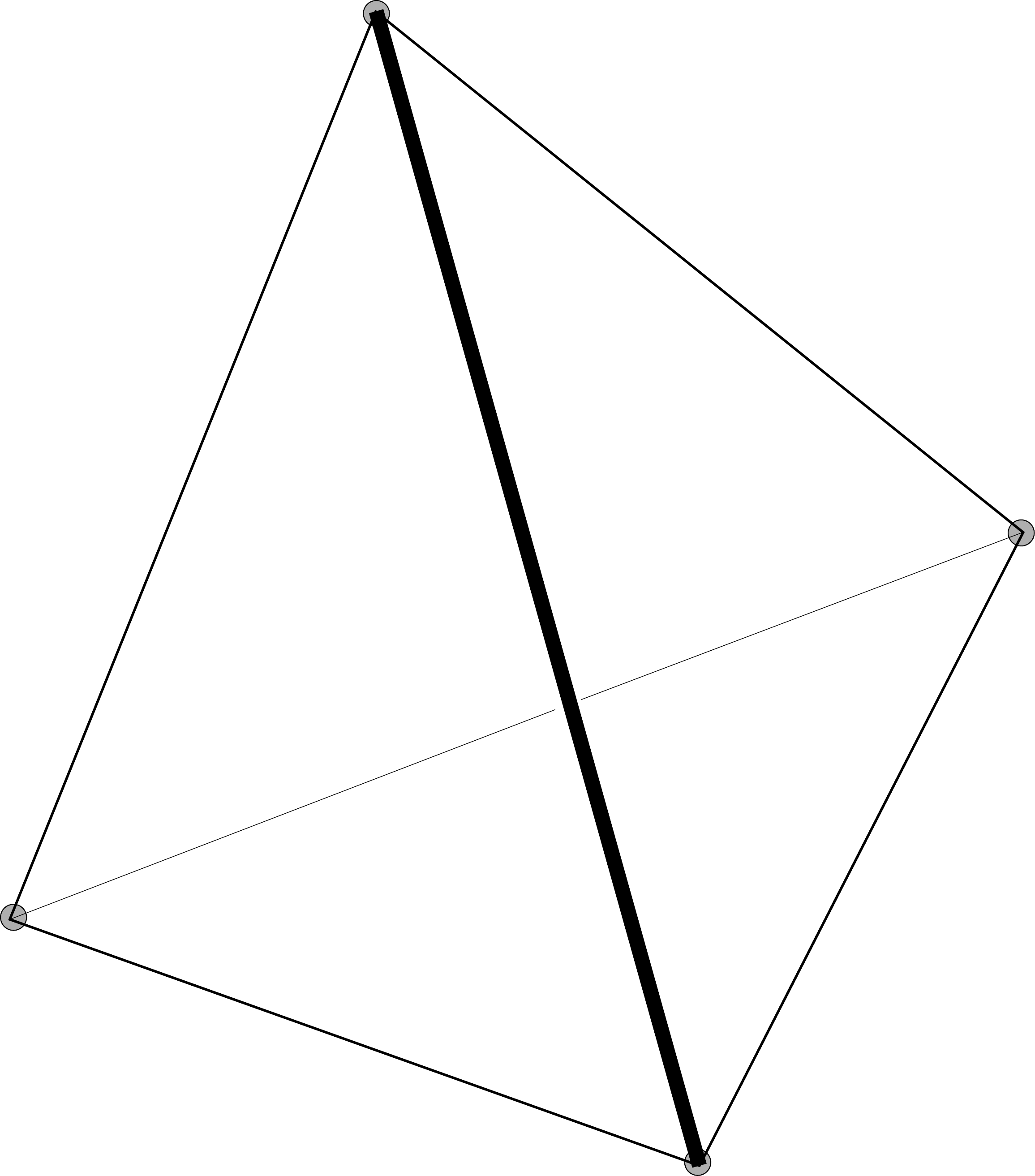}
	\subcaption{An arc in $V_1$.}
	\label{tetra pic 1}
	\end{minipage} \hfill
	\begin{minipage}{.3\textwidth}
	\centering
	\includegraphics[width=4cm]{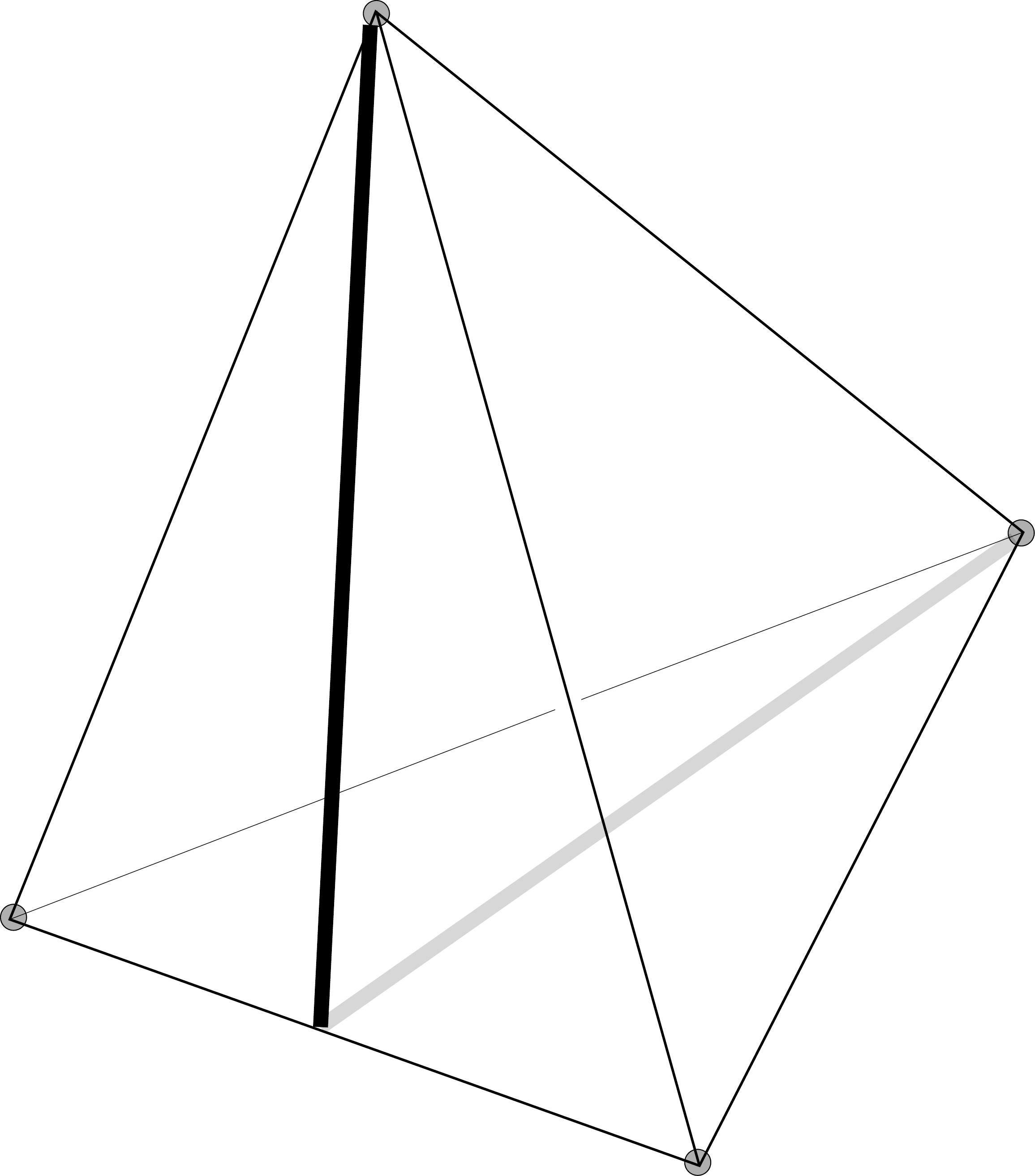}
	\subcaption{An arc in $V_2$.}
	\label{tetra pic 2}
	\end{minipage} \hfill
	\begin{minipage}{.3\textwidth}
	\centering
	\includegraphics[width=4cm]{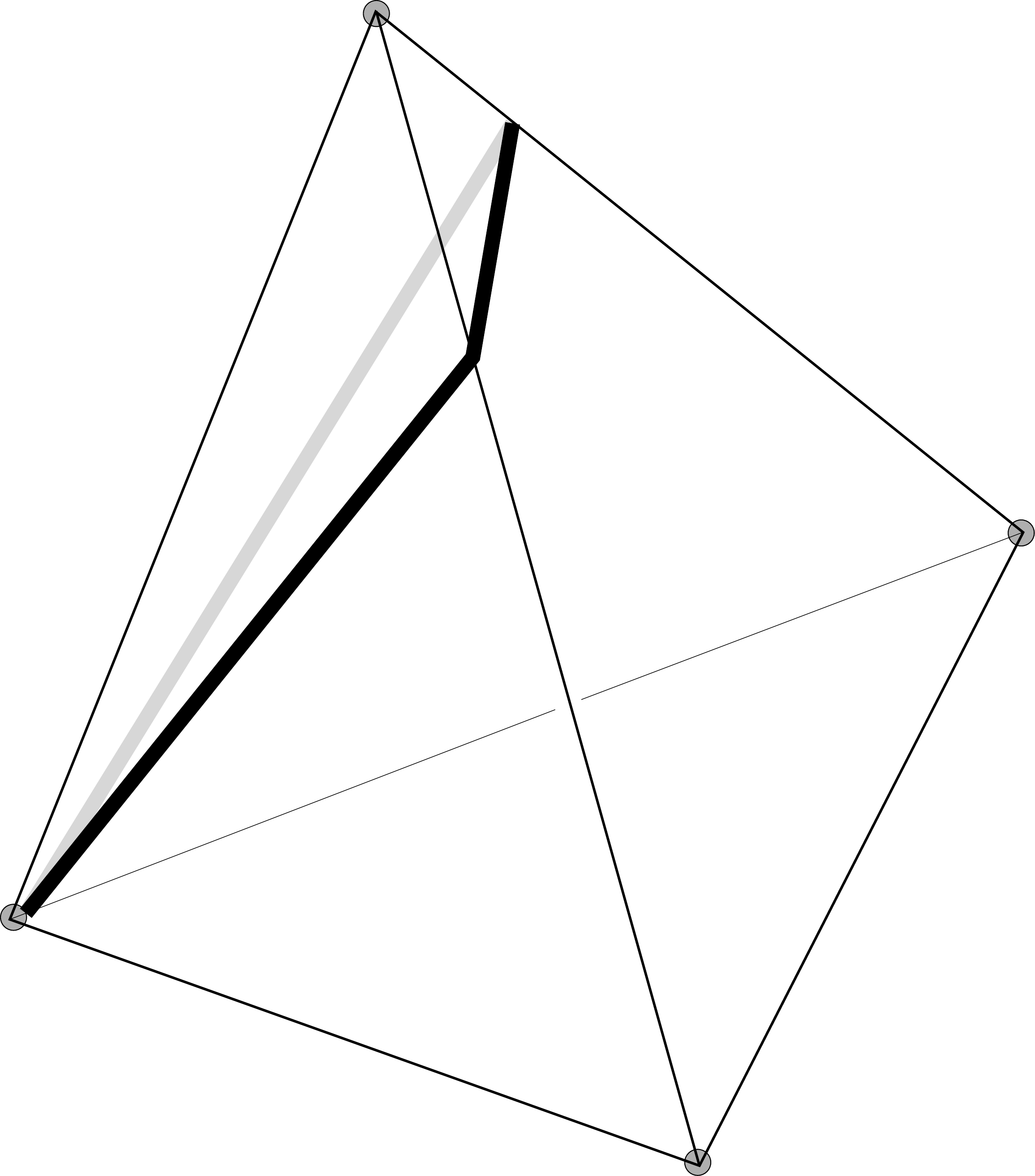}
	\subcaption{An arc in $V_3$.}
	\label{tetra pic 2}
	\end{minipage}
	\caption{A schematic of the arcs $V \subset \ca(\Sigma_4)$.}
	\label{tetra pics}
\end{figure}

Let $I$ be a maximal independent set of $G$, and let $a_i = |I\cap V_i|$ for $i=1,2,3$.
It is straightforward to check that $(a_1,a_2,a_3) \in \{(1,1,0), (1,0,2), (0,1,4), (0,2,2), (0,3,3) \}$.
A function $f:V \to \mathbb{R}_{\ge 0}$ satisfying $f|_{V_i} = c_i$, for $i=1,2,3$, is a fractional clique if $(c_1,c_2,c_3) \cdot (a_1,a_2,a_3) \le 1$ for all $(a_1,a_2,a_3)$ as above.
Thus, taking $c_1=7/9$, $c_2=2/9$, $c_3=1/9$ yields a fractional clique of total value $\sum c_i |V_i| =6c_1 + 6c_2 +12c_3=22/3$, and the lemma follows.
\end{proof}

\begin{Prop}
\label{prop:four-holed sep}
$\chi(\ca\cs(\Sigma_4))=\omega(\ca\cs(\Sigma_4))=3$.
\end{Prop}

\begin{proof}
Mark one of the holes.
One of the components in the complement of a separating arc $a$ contains a single hole of $\Sigma_4$.
If it is not marked, let $c(a)$ denote this hole; if it is, let $c(a)$ denote the hole containing the endpoints of $a$.
The reader may check that $c$ defines a proper 3-coloring, and it is easy to locate a clique of size 3.
\end{proof}

\begin{Prop}
\label{prop:four-holed nosep}
$\chi(\ca\cn(\Sigma_4))=\omega(\ca\cn(\Sigma_4))=6$.
\end{Prop}

\begin{proof}
Write $\del \Sigma_4 = \del_1 \cup \del_2 \cup \del_3 \cup \del_4$.
Let $\ca_{ij}(\Sigma_4)$ denote the subgraph induced on arcs with one endpoint on $\del_i$ and the other on $\del_j$.
Note that an arc in $\ca_{12}(\Sigma_4)$ is disjoint from a unique arc in $\ca_{34}(\Sigma_4)$.
Fix such a pair $(a_{12},a_{34})$.
Using Theorem \ref{thm: arcs unique}, 2-color both of $\ca_{12}(\Sigma_4)$ and $\ca_{34}(\Sigma_4)$ so that $a_{12}$ and $a_{34}$ get opposite colors.
One checks that any arc in $\ca_{12}(\Sigma_4)$ gets the opposite color from the unique arc that it is disjoint from in $\ca_{34}(\Sigma_4)$.
Thus, we obtain a proper 2-coloring of the subgraph induced on $\ca_{12}(\Sigma_4) \cup \ca_{34}(\Sigma_4)$.
Copy this coloring onto the other two subgraphs corresponding to the partitions of the holes into two pairs, using a different palette for each.
The result is the desired 6-coloring.
The arcs in $V_1$ described in the proof of Lemma \ref{lem: frac num four-holed} form a clique of size 6.
\end{proof}

\begin{Thm}
\label{thm:four-holed arcs}
$8 \le \chi(\ca(\Sigma_4)) \le 9$.
\end{Thm}

\begin{proof}
Immediate from Lemma \ref{lem: frac num four-holed} and Propositions \ref{prop:four-holed sep} and \ref{prop:four-holed nosep}.
\end{proof}

As a first step towards improving the upper bound in Theorem \ref{thm: planar arc graph intro}, it is possible to color a subcomplex of the nonseparating arcs with fewer colors than those used above. Let $\ca' \subset \ca(\Sigma_4)$ denote the subcomplex induced on arcs with exactly one endpoint on a fixed hole $\del_i$, and note that six colors are used to color $\ca'$ in the coloring of Proposition \ref{prop:four-holed nosep}.
We demonstrate in Theorem \ref{thm: chi(a)=4} that four suffice.

We begin by relating $\ca'$ to $\cc(\Sigma_4)$.
By definition, a simplex in $\cc(\Sigma_4)$ consists of curves with pairwise minimal intersection number 2.
There exists a well-known isomorphism between the complex $\cc(\Sigma_4)$ and the Farey complex $\cf$ \cite[p.~94-95]{Farbprimer}.
The vertex set of $\cf$ is $P^1(\bz^2)$, the set of lines in $\bz^2$.
Given a pair of lines $L_1,L_2 \in P^1(\bz^2)$, let $d(L_1,L_2)$ denote the index in $\bz^2$ of the subgroup generated by their elements.
The edge set of $\cf$ consists of all pairs $(L_1,L_2)$ satisfying $d(L_1,L_2)=1$.
The group $\mathrm{PSL}(2,\bz)$ acts on $P^1(\bz^2)$ and by extension on $\cf$.

Form the pure, 3-dimensional supercomplex $\cf' \supset \cf$ whose 3-simplices are the unions of the 2-simplices in $\cf$ that share a common edge.
Given $(L_1,L_2) \in E(\cf)$, with $L_i$ generated by $x_i \in \bz^2$, there exist two maximal simplices in $\cf$ containing $(L_1,L_2)$, and they take the form $(L_1,L_2,L^\pm)$, where $L^\pm$ is generated by $x_1 \pm x_2$.
It follows that $d(L^+,L^-)=2$.
Thus, for every edge $(L,L') \in E(\cf')$, we have $d(L,L') \in \{1,2\}$.
Observe that $\mathrm{PSL}(2,\bz)$ extends to an action on $\cf'$.

\begin{Prop}
$\ca' \approx \cf'$.
\end{Prop}

\begin{proof}
An arc in $\ca'$ is disjoint from a unique curve in $\cc(\Sigma_4)$, and vice versa.
In this way, we obtain a natural bijection between the vertex sets of these complexes.
Moreover, disjoint arcs in $\ca'$ on different endpoint pairs correspond to curves in $\cc(\Sigma_4)$ with minimal geometric intersection number 2, and vice versa.
It follows that $\cf \approx \cc(\Sigma_4)$ naturally embeds as a subcomplex of $\ca'$.
The 2-simplices of $\ca'$ contained in $\cf$ are spanned by triples of pairwise disjoint arcs in $\ca'$ on different pairs of endpoints.
Given two such 2-simplices that share a common edge, the pair of vertices in these simplices not on the edge are disjoint arcs in $\ca'$ with the same pair of endpoints.
Therefore, the 4 vertices in the union of these 2-simplices span a 3-simplex in $\ca'$.
It follows that $\cf' \subset \ca'$.
Conversely, a maximal simplex $\sigma \subset \ca'$ consists of 4 pairwise disjoint arcs, precisely two of which have the same endpoint pair.
Thus, it contains two 2-simplices in $\cf$ that abut along an edge, so $\sigma \subset \cf'$.
It follows that $\ca' \subset \cf'$, and the proof is complete.
 \end{proof}

Recall that for a positive integer $n$, the congruence subgroup $\Gamma(n) \subset \mathrm{PSL}(2,\bz)$ is the kernel of the natural epimorphism $\mathrm{PSL}(2,\bz) \to \mathrm{PSL}(2,\bz/n\bz)$ obtained by reducing$\pmod n$.

\begin{Thm}
\label{thm: chi(a)=4}
$\chi(\ca')=\chi(\cf')=4$.  Moreover, the $\Gamma(3)$-orbits under the action by $\mathrm{PSL}(2,\bz)$ on $P^1(\bz^2)$ comprise the color classes in a proper 4-coloring of $\cf'$.
\end{Thm}

\begin{proof}
Since $\cf'$ is 3-dimensional, $\chi(\cf') \ge 4$ follows at once.
Next, map a line in $P^1(\bz^2)$ to its$\pmod 3$ reduction in $P^1((\bz/3\bz)^2)$.
Observe that $P^1((\bz/3\bz)^2)$ consists of four lines.
Given an edge $(L_1,L_2) \in E(\cf')$, $d(L_1,L_2)$ is relatively prime to 3, so the subgroup generated by the elements of $L_1$ and $L_2$, reduced$\pmod 3$, is all of $(\bz/3\bz)^2$.  In particular, $L_1$ and $L_2$ reduce to distinct lines$\pmod 3$.
Therefore, the$\pmod 3$ reduction map defines a proper 4-coloring of $\cf'$, and $\chi(\cf') = 4$, as desired.

Next, the transitive action by $\mathrm{PSL}(2,\bz)$ on $\cf'$ permutes the color classes under the 4-coloring $f$ just described.
Let $g \in \mathrm{PSL}(2,\bz)$ and select a line $L \in P^1(\bz^2)$.
We have $f(g \cdot L) = \overline{g} \cdot f(L)$, where $\overline{g} \in PSL(2,\bz/3\bz)$ denotes the reduction of $g\pmod3$.
It follows that the subgroup of $\mathrm{PSL}(2,\bz)$ that preserves the color classes consists of those group elements $g$ for which $\overline{g}$ fixes all lines in $(\bz/3\bz)^2$.
Such an element $\overline{g}$ is represented by a diagonal matrix in $\mathrm{SL}(2,\bz/3\bz)$, which forces $\overline{g} = \pm I$.
Therefore, the color-preserving subgroup of $\mathrm{PSL}(2,\bz)$ is precisely $\Gamma(3)$, which completes the proof.
\end{proof}

In the language of this proof, the unique 3-coloring of $\cf$ comes by mapping each line to its$\pmod 2$ reduction; the color classes are the $\Gamma(2)$-orbits.


\section{The genus two surface}
\label{sec: genus 2}

In this section, we study the chromatic number of the curve graph of the closed surface of genus two and prove Theorem \ref{thm: S_2}.
We first consider a pair of finite graphs that we use to obtain the required lower bounds.
We then apply a homomorphism out of $\cn(S_2)$ defined using hyperbolic geometry in order to obtain the required upper bounds.
This homomorphism admits an alternative description in terms of homology, as we show in Proposition \ref{prop: kneser symplectic}.

\begin{figure}
\includegraphics[width=2in]{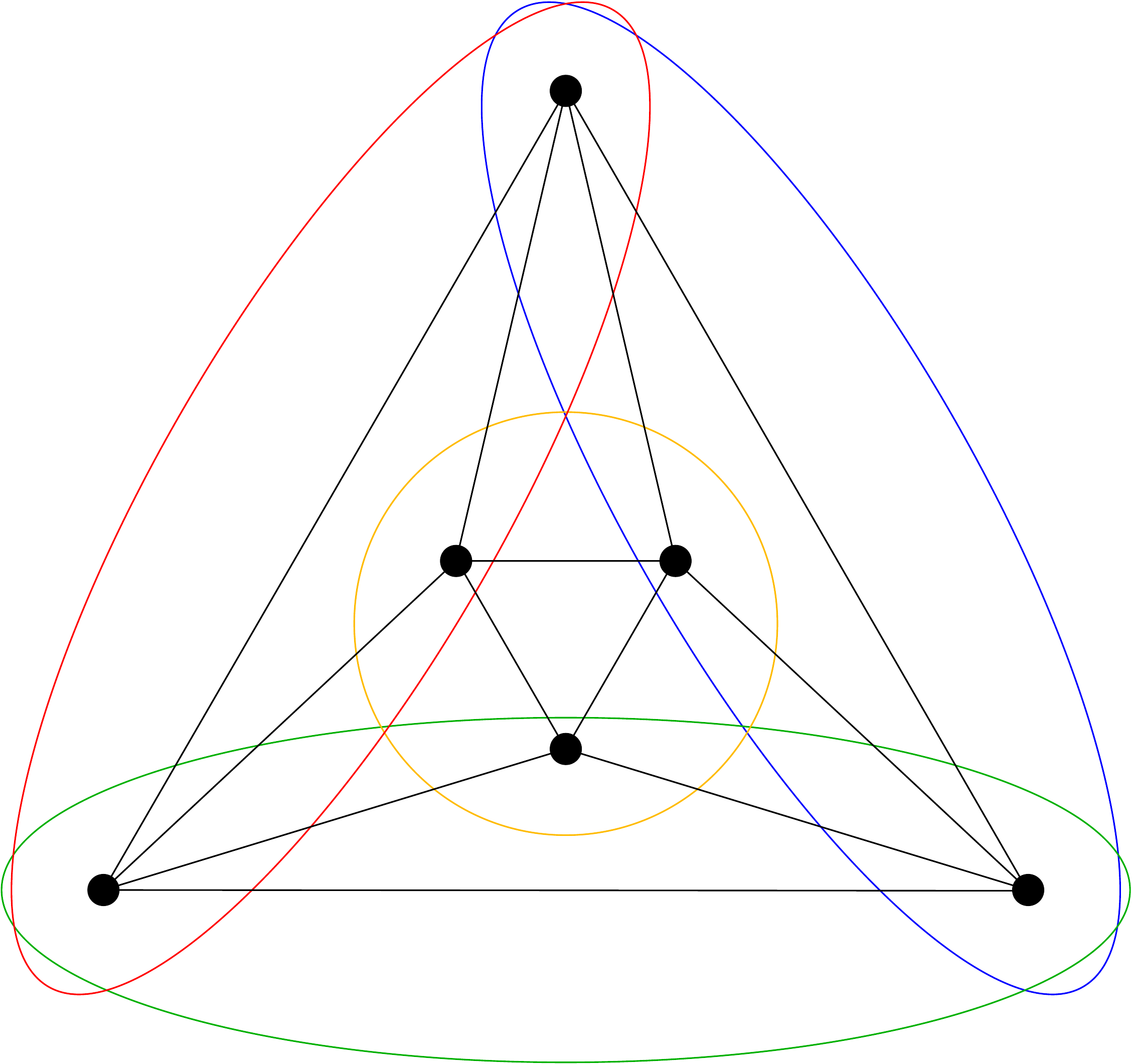}
\caption{The octahedron and four bisecting circles, colored solely for visual aid.}
\label{fig: octahedron}
\end{figure}

Consider the octahedron $O$ and the four circles appearing in Figure \ref{fig: octahedron}.
The four circles separate the four pairs of antipodal faces in $O$.
Let $E$ denote the set of edges of $O$ and $C$ the set of circles.
Let $\cc$ denote the graph on $E \cup C$, where adjacency connotes disjointness, and let $\cn$ denote the subgraph of $\cc$ induced on $E$.
Thus, $\cn$ is the complement of the line graph of $O$.

\begin{Prop}
\label{prop: octahedron}
$\chi(\cn) = 4$ and $\chi(\cc) = 5$.
\end{Prop}

\begin{proof}
The fact that $\chi(\cn)=4$ is a well-known and simple exercise in graph theory.
We obtain $\chi(\cc) \le 5$ by properly 4-coloring $\cn$ and giving every circle a common fifth color.
The reader may check that an independent set in $\cc$ contains at most 4 elements, with equality if and only if they are 4 edges meeting at a vertex or 4 circles.
As $|E \cup C| = 16$, if there were a 4-coloring of $\cc$, then each color class must contain 4 elements.
Thus, one color class must consist of the 4 circles, so the remaining ones induce a 3-coloring of $\cn$, a contradiction.
Thus, $\chi(\cc) \ge 5$, completing the proof.
\end{proof}

\begin{proof}[Proof of Theorem \ref{thm: S_2}]
Place a hyperbolic structure on $S_2$.  It admits a unique hyperelliptic involution $J$ with 6 fixed points (see the discussions following \cite[Thm. 3.10 \& Prop. 7.15]{Farbprimer}).
Every non-separating geodesic $\alpha \subset S_2$ meets $\mathrm{Fix}(J)$ in precisely two points \cite[Proposition 2.3]{MalesteinDesigns}.
This result follows at once from \cite[Theorem 1]{Haashyperelliptic}, according to which $J(\alpha) = \alpha$ and $J$ reverses the orientation on $\alpha$: it follows that $J|\alpha$ is a reflection, so it has two fixed points.
Clearly, disjoint non-separating geodesics must meet $\mathrm{Fix}(J)$ in distinct pairs of points.
Therefore, the assignment of a non-separating geodesic to the pair of points it meets in $\mathrm{Fix}(J)$ defines a homomorphism $h: \cn(S_2) \to \kg(6,2)$.
Since $\chi(\kg(6,2))=4$, $h$ induces a proper 4-coloring of $\cn(S_2)$.
Any two curves in $\cs(S_2)$ intersect, so giving them a common fifth color leads to a proper 5-coloring of $\cc(S_2)$.
Thus, we obtain the required upper bounds $\chi(\cn(S_2)) \le 4$ and $\chi(\cc(S_2)) \le 5$.

Next, form the double-cover of the two-sphere branched along the 6 vertices of $O$.
The result is a surface homeomorphic to $S_2$.
The edges in $E$ lift to non-separating curves, and the circles in $C$ lift to separating curves.
The subgraph of $\cc(S_2)$ induced on their lifts is isomorphic to $\cc$, while the subgraph induced on the lifts of $E$ is isomorphic to $\cn$, leading to the required lower bounds $\chi(\cn(S_2)) \ge 4$ and $\chi(\cc(S_2)) \ge 5$.
\end{proof}

We now turn to a non-geometric description of the map $h$.
The$\pmod 2$ intersection pairing $\widehat \iota$ equips $H_1(S_g;\bF_2)$ with the structure of a symplectic vector space.
Let $\sp(2g)$ denote the graph on the non-zero elements of $H_1(S_g;\bF_2)$, where two distinct elements span an edge if they are orthogonal with respect to $\widehat \iota$.
There exists a natural map $f: \cn(S_g) \to \sp(2g)$ assigning a curve to its$\pmod2$ homology class.
This map is a relative of the natural homomorphism $\cc(\Sigma_n) \to \kg(n)$ studied in the case of a planar surface. 
The map $f$ is a homomorphism when $g=2$, since distinct homologous curves intersect, but it is not for any value $g \ge 3$.

\begin{Prop}
\label{prop: kneser symplectic}
The homomorphisms $f : \cn(S_2) \to \sp(4)$ and $h : \cn(S_2) \to \kg(6,2)$ coincide in the sense that there exists an isomorphism $\varphi: \sp(4) \to \kg(6,2)$ such that $\varphi \circ f = h$.
In particular, up to automorphism of $\kg(6,2)$, $h$ does not depend on the choice of hyperbolic structure on $S_2$.
\end{Prop}

The existence of an isomorphism $\sp(4) \approx \kg(6,2)$ follows from the fact that both are strongly regular graphs with parameters $(15,6,1,3)$, of which there exists a unique isomorphism type \cite[Theorems (2.3) and (4.14)]{CvL}, \cite[Remark after Theorem 21.5]{vLW}.
Stripping the proof of Proposition \ref{prop: kneser symplectic} to its algebraic core establishes this isomorphism directly.

\begin{proof}
For context, we begin in somewhat greater generality.
Fix $g \ge 0$ and let $P=\{p_1,\dots,p_{2(g+1)}\}$ denote the set of marked points on a sphere $S$.
Define a bilinear pairing $b$ on $H_0(P;\bF_2)$ by the rule $b([p_i],[p_j]) =\delta_{ij}$.
The long exact sequence in homology of the pair $(S,P)$ leads to a natural identification $H_1(S,P;\bF_2) \approx [P]^\perp$, to which $b$ restricts and has annihilator generated by $[P] = \sum_i [p_i]$.
Thus, $b$ descends to a non-degenerate bilinear pairing $\overline{b}$ on $[P]^\perp / ([P])$.
Let $\Sigma(S,P) \approx S_g$ denote the double cover of $S$ branched along $P$.
The covering induces a transfer homomorphism between inner product spaces $\tau: (H_1(S,P;\bF_2),b) \to (H_1(\Sigma(S,P);\bF_2),\widehat{\iota})$.
One checks that $\tau$ surjects and $\ker(\tau) = ([P])$, leading to a natural identification $([P]^\perp / ([P]), \overline{b}) \approx (H_1(\Sigma(S,P);\bF_2),\widehat{\iota})$.

Now specialize to the case $g=2$.  Each non-zero class in $[P]^\perp / ([P])$ is uniquely represented by an element of the form $[p_i]+[p_j]$ for distinct $i,j\in\{1,\dots,6\}$, and two such elements are $\overline{b}$-orthogonal if and only if they are equal or correspond to disjoint 2-element subsets.
Let $\varphi : \kg(6,2) \to \sp(4)$ denote the map assigning the subset $\{i,j\}$ to $\tau([p_i]+[p_j])$.
The preceding remarks show that $\varphi$ is an isomorphism.

Lastly, select a non-separating, simple, closed geodesic $\alpha \subset S_2$.
It meets $\mathrm{Fix}(J)$ in two points.
Let $P$ denote the points covered by $\mathrm{Fix}(J)$.
The image of $\alpha$ in the quotient $S_2 / J$ is a simple arc $a$ that meets $P$ precisely in its endpoints $p_i$ and $p_j$.
Thus, $h(\alpha) = \{i,j\}$, and $\varphi(h(\alpha)) = [\alpha] = f(\alpha)$.
Therefore, $\varphi \circ f = h$.
\end{proof}


\section{Problems for further study}
\label{sec: conclusion}

Here we collect some problems for further study.
The first one is the most prominent:

\begin{Prob}
\label{p: cc(s_g)}
\textup{
Improve the estimates on the (fractional) chromatic numbers of $\cc(S_g)$ and $\ca(\Sigma_n)$ in Theorems \ref{thm: s_g intro} and \ref{thm: planar arc graph intro}.  We believe that both are closer to the stated lower bounds.
}
\end{Prob}

\begin{Prob}
\label{p: kg(n)}
\textup{
Determine the exact value of $\chi(\kg(n))$.  Does it equal the upper bound given in Theorem \ref{thm: kneser chromatic}?
}
\end{Prob}

\begin{Prob}
\label{p: kg(n)}
\textup{
What is $\chi(\ca(\Sigma_4))$?  It is 8 or 9 by Theorem \ref{thm:four-holed arcs}.
}
\end{Prob}

\begin{Prob}
\textup{
Generalize the results on unique colorability.
For instance, does there exist a unique homomorphism $\cc(\Sigma_n) \to \kg(n) \smallsetminus \kg(n,1)$?
Compare this question with the rigidity of embeddings $\cg(n) \smallsetminus \cg(n,1) \into \cc(\Sigma_n)$ \cite{Aramayonafinite}.
Does there exist a unique homomorphism $\cn(S_2) \to \kg(6,2)$?
Proposition \ref{prop: kneser symplectic} provides evidence for this possibility.
}
\end{Prob}

\begin{Prob}
\textup{
Explore other graph-theoretic properties of the curve graphs related to the chromatic number, such as the Shannon capacity and spectra.
}
\end{Prob}

\begin{Prob}
\textup{
Identify $\Aut(\cc_v(S))$ with a subgroup of $\Mod(S)$ (cf.~\cite{Ivanovautomorphisms}). 
Is it isomorphic to the stabilizer of $v$ in $\Mod(S)$?
Does there exist a simple generating set for this stabilizer analogous to the Humphries generating set?
In that way, it might be possible to recover the closed case of Proposition \ref{prop: connected} without passage to the case of non-empty boundary.
}
\end{Prob}

\begin{Prob}
\textup{
Describe a generating set for the kernel of the Chillingworth homomorphism.
Is it generated by the Johnson kernel and $(g-1)$-th powers of bounding pair maps?
Is it finitely generated?
}
\end{Prob}

\bibliographystyle{alpha}
\bibliography{biblio}

\newcommand{\etalchar}[1]{$^{#1}$}
\begin{thebibliography}{FMFPH{\etalchar{+}}09}

\bibitem[AL13]{Aramayonafinite}
Javier Aramayona and Christopher~J. Leininger.
\newblock Finite rigid sets in curve complexes.
\newblock {\em J. Topol. Anal.}, 5(2):183--203, 2013.

\bibitem[BBF15]{Bestvinaquasitree}
Mladen Bestvina, Ken Bromberg, and Koji Fujiwara.
\newblock Constructing group actions on quasi-trees and applications to mapping
  class groups.
\newblock {\em Publ. Math. Inst. Hautes \'Etudes Sci.}, 122(1):1--64, 2015.

\bibitem[BBM15]{Birmanfinite}
Joan Birman, Nathan Broaddus, and William Menasco.
\newblock Finite rigid sets and homologically nontrivial spheres in the curve
  complex of a surface.
\newblock {\em J. Topol. Anal.}, 7(1):47--71, 2015.

\bibitem[BCM12]{Brockclassification}
Jeffrey Brock, Richard Canary, and Yair Minsky.
\newblock The classification of {K}leinian surface groups, {II}: {T}he ending
  lamination conjecture.
\newblock {\em Ann. of Math. (2)}, 176(1):1--149, 2012.

\bibitem[BM15]{Birmandeadends}
Joan Birman and William Menasco.
\newblock The curve complex has dead ends.
\newblock {\em Geom. Dedicata}, 177:71--74, 2015.

\bibitem[Bow08]{Bowditchtight}
Brian~H. Bowditch.
\newblock Tight geodesics in the curve complex.
\newblock {\em Invent. Math.}, 171(2):281--300, 2008.

\bibitem[BP07]{BuserDistribution}
Peter Buser and Hugo Parlier.
\newblock The distribution of simple closed geodesics on a {R}iemann surface.
\newblock In {\em Complex analysis and its applications}, volume~2 of {\em
  OCAMI Stud.}, pages 3--10. Osaka Munic. Univ. Press, Osaka, 2007.

\bibitem[Bus92]{BuserGeometry}
Peter Buser.
\newblock {\em Geometry and spectra of compact {R}iemann surfaces}, volume 106
  of {\em Progress in Mathematics}.
\newblock Birkh{\"a}user Boston, Inc., Boston, MA, 1992.

\bibitem[Chi72a]{ChillingworthwindingI}
D.~R.~J. Chillingworth.
\newblock Winding numbers on surfaces, {I}.
\newblock {\em Math. Ann.}, 196(3):218--249, 1972.

\bibitem[Chi72b]{ChillingworthwindingII}
D.~R.~J. Chillingworth.
\newblock Winding numbers on surfaces. {II}.
\newblock {\em Math. Ann.}, 199(3):131--153, 1972.

\bibitem[CvL91]{CvL}
Peter~J. Cameron and J.~H. van Lint.
\newblock {\em Designs, graphs, codes and their links}.
\newblock Cambridge University Press, Cambridge, 1991.

\bibitem[FM11]{Farbprimer}
Benson Farb and Dan Margalit.
\newblock {\em A Primer on Mapping Class Groups}.
\newblock Princeton University Press, 2011.

\bibitem[FMFPH{\etalchar{+}}09]{FabilaChromatic}
Ruy Fabila-Monroy, David Flores-Pe{\~n}aloza, Clemens Huemer, Ferran Hurtado,
  Jorge Urrutia, and David~R Wood.
\newblock On the chromatic number of some flip graphs.
\newblock {\em Discrete Math. Theor. Comput. Sci.}, 11(2):47--56, 2009.

\bibitem[GR13]{Godsilgraphs}
Chris Godsil and Gordon~F. Royle.
\newblock {\em Algebraic graph theory}, volume 207.
\newblock Springer-Verlag, New York, 2013.

\bibitem[Gup]{Radhika}
Radhika Gupta.
\newblock Private communication.

\bibitem[Har81]{Harveyboundary}
W.~J. Harvey.
\newblock Boundary structure of the modular group.
\newblock In {\em Riemann surfaces and related topics: {P}roceedings of the
  1978 {S}tony {B}rook {C}onference ({S}tate {U}niv. {N}ew {Y}ork, {S}tony
  {B}rook, {N}.{Y}., 1978)}, volume~97 of {\em Ann. of Math. Stud.}, pages
  245--251. Princeton Univ. Press, Princeton, N.J., 1981.

\bibitem[HS89]{Haashyperelliptic}
Andrew Haas and Perry Susskind.
\newblock The geometry of the hyperelliptic involution in genus two.
\newblock {\em Proc. Amer. Math. Soc.}, 105(1):159--165, 1989.

\bibitem[Irm15]{Irmerchillingworth}
Ingrid Irmer.
\newblock The {C}hillingworth class is a signed stable length.
\newblock {\em Algebr. Geom. Topol.}, 15(4):1863--1876, 2015.

\bibitem[Iva97]{Ivanovautomorphisms}
Nikolai~V. Ivanov.
\newblock Automorphisms of complexes of curves and of {T}eichm\"uller spaces.
\newblock In {\em Progress in knot theory and related topics}, volume~56 of
  {\em Travaux en Cours}, pages 113--120. Hermann, Paris, 1997.

\bibitem[JMM96]{Juvansystems}
M.~Juvan, A.~Malni{\v{c}}, and B.~Mohar.
\newblock Systems of curves on surfaces.
\newblock {\em J. Combin. Theory Ser. B}, 68(1):7--22, 1996.

\bibitem[Joh80]{Johnsonabelian}
Dennis Johnson.
\newblock An abelian quotient of the mapping class group $\mathcal{I}_g$.
\newblock {\em Math. Ann.}, 249(3):225--242, 1980.

\bibitem[Joh83]{Johnsonstructure1}
Dennis Johnson.
\newblock The structure of the {T}orelli group {I}: a finite set of generators
  for $\mathcal{J}$.
\newblock {\em Ann. of Math. (2)}, 118(3):423--442, 1983.

\bibitem[Joh85]{Johnsonstructure2}
Dennis Johnson.
\newblock The structure of the {T}orelli group {II}: A characterization of the
  group generated by twists on bounding curves.
\newblock {\em Topology}, 24(2):113--126, 1985.

\bibitem[KK13]{Kim-Koberda}
Sang-hyun Kim and Thomas Koberda.
\newblock Right-angled {A}rtin groups and finite subgraphs of curve graphs.
\newblock {\em {\tt arXiv:1310.4850}}, 2013.

\bibitem[Kne55]{Kneseroriginal}
Martin Kneser.
\newblock Aufgabe 360.
\newblock {\em Jahresbericht der Deutschen Mathematiker-Vereinigung}, 2:27,
  1955.

\bibitem[KP02]{Kentannularbraids}
Richard~P. Kent\phantom{ }IV and David Peifer.
\newblock A geometric and algebraic description of annular braid groups.
\newblock {\em Internat. J. Algebra Comput.}, 12:85--97, 2002.

\bibitem[Lee89]{Leeassociahedron}
Carl~W. Lee.
\newblock The associahedron and triangulations of the $n$-gon.
\newblock {\em European J. Combin.}, 10(6):551--560, 1989.

\bibitem[Lov78]{LovaszKneser}
L{\'a}szl{\'o} Lov{\'a}sz.
\newblock Kneser's conjecture, chromatic number, and homotopy.
\newblock {\em J. Combin. Theory Ser. A}, 25(3):319--324, 1978.

\bibitem[Min10]{Minskyclassification}
Yair Minsky.
\newblock The classification of {K}leinian surface groups. {I}. {M}odels and
  bounds.
\newblock {\em Ann. of Math. (2)}, 171(1):1--107, 2010.

\bibitem[MM99]{Masurgeometry1}
Howard~A. Masur and Yair~N. Minsky.
\newblock Geometry of the complex of curves. {I}. {H}yperbolicity.
\newblock {\em Invent. Math.}, 138(1):103--149, 1999.

\bibitem[MP78]{MeeksPatrusky}
William~H. Meeks, III and Julie Patrusky.
\newblock Representing homology classes by embedded circles on a compact
  surface.
\newblock {\em Illinois J. Math.}, 22(2):262--269, 1978.

\bibitem[MRT14]{MalesteinDesigns}
Justin Malestein, Igor Rivin, and Louis Theran.
\newblock Topological designs.
\newblock {\em Geom. Dedicata}, 168:221--233, 2014.

\bibitem[Prz15]{PrzytyckiArcs}
Piotr Przytycki.
\newblock Arcs intersecting at most once.
\newblock {\em Geom. Funct. Anal.}, 25(2):658--670, 2015.

\bibitem[Put08]{Putmanconnectivity}
Andrew Putman.
\newblock A note on the connectivity of certain complexes associated to
  surfaces.
\newblock {\em Enseign. Math. (2)}, 54(3-4):287--301, 2008.

\bibitem[Put09]{Putmanjohnson}
Andrew Putman.
\newblock The {J}ohnson homomorphism and its kernel.
\newblock {\em J. Reine Angew. Math.}, to appear, 2009.

\bibitem[Sar11]{Sarkar2011}
Sucharit Sarkar.
\newblock Maslov index formulas for {W}hitney $n$-gons.
\newblock {\em J. Symp. Geom.}, 9(2):251--270, 2011.

\bibitem[STT88]{Sleatorrotation}
Daniel~D. Sleator, Robert~E. Tarjan, and William~P. Thurston.
\newblock Rotation distance, triangulations, and hyperbolic geometry.
\newblock {\em J. Amer. Math. Soc.}, 1(3):647--681, 1988.

\bibitem[vLW01]{vLW}
J.~H. van Lint and R.~M. Wilson.
\newblock {\em A course in combinatorics}.
\newblock Cambridge University Press, Cambridge, 2001.

\end{thebibliography}

\end{document}